\documentclass{scrartcl}

\usepackage{amsmath}
\usepackage{amsfonts}
\usepackage{amssymb}
\usepackage{graphicx}
\usepackage{hyperref}

\newcommand{\bbbc}{\mathbb{C}}
\newcommand{\bbbn}{\mathbb{N}}
\newcommand{\bbbr}{\mathbb{R}}

\newcommand{\Idx}{\mathcal{I}}
\newcommand{\Jdx}{\mathcal{J}}
\newcommand{\ctI}{\mathcal{T}_{\Idx}}
\newcommand{\lfI}{\mathcal{L}_{\Idx}}
\newcommand{\ctII}{\mathcal{T}_{\Idx\times\Idx}}
\newcommand{\lfII}{\mathcal{L}_{\Idx\times\Idx}}
\newcommand{\lfaII}{\mathcal{L}^+_{\Idx\times\Idx}}
\newcommand{\lfiII}{\mathcal{L}^-_{\Idx\times\Idx}}

\newcommand{\diam}{\mathop{\operatorname{diam}}\nolimits}
\newcommand{\dist}{\mathop{\operatorname{dist}}\nolimits}
\newcommand{\supp}{\mathop{\operatorname{supp}}\nolimits}
\newcommand{\treeroot}{\mathop{\operatorname{root}}\nolimits}
\newcommand{\chil}{\mathop{\operatorname{chil}}\nolimits}
\newcommand{\desc}{\mathop{\operatorname{desc}}\nolimits}
\newcommand{\dirchil}{\mathop{\operatorname{dirchil}}\nolimits}
\newcommand{\dirblock}{\mathop{\operatorname{dirblock}}\nolimits}

\newtheorem{theorem}{Theorem}
\newtheorem{definition}[theorem]{Definition}
\newtheorem{lemma}[theorem]{Lemma}
\newtheorem{remark}[theorem]{Remark}
\newenvironment{proof}{\emph{Proof.}}{$\strut$\hfill$\Box$}

\title{Memory-efficient compression of $\mathcal{DH}^2$-matrices
  for high-frequency Helmholtz problems}
\author{Steffen B\"orm and Janne Henningsen
\thanks{Mathematisches Seminar, Christian-Albrechts-Universit\"at Kiel,
  Heinrich-Hecht-Platz 6, Kiel, Germany, EMail \href{mailto:boerm@math.uni-kiel.de}{\texttt{boerm@math.uni-kiel.de}}}}

\begin{document}

\maketitle

\begin{abstract}
Directional interpolation is a fast and efficient compression technique
for high-frequency Helmholtz boundary integral equations, but requires
a very large amount of storage in its original form.
Algebraic recompression can significantly reduce the storage requirements
and speed up the solution process accordingly.
During the recompression process, weight matrices are required to
correctly measure the influence of different basis vectors on the final
result, and for highly accurate approximations, these weight matrices
require more storage than the final compressed matrix.

We present a compression method for the weight matrices and
demonstrate that it introduces only a controllable error to the
overall approximation.
Numerical experiments show that the new method leads to a
significant reduction in storage requirements.
\end{abstract}

\section{Introduction}

We consider boundary element discretizations of the Helmholtz equation
\begin{equation*}
  -\Delta u - \kappa^2 u = 0
\end{equation*}
with the \emph{wave number} $\kappa\in\bbbr$ on a domain
$\Omega\subseteq\bbbr^3$.
Using the fundamental solution
\begin{align}\label{kernel}
  g(x,y) &= \frac{\exp(\iota \kappa \|x-y\|)}{4\pi \|x-y\|} &
  &\text{ for all } x,y\in\bbbr^3,\ x\neq y,
\end{align}
the boundary integral formulation leads to an equation of the form
\begin{align}\label{eq:integral}
  \int_{\partial\Omega} g(x,y)\, \frac{\partial u}{\partial n}(y) \,dy
  &= \frac{1}{2} u(x)
   + \int_{\partial\Omega} \frac{\partial g}{\partial n_y}(x,y) u(y) \,dy &
  &\text{ for all } x\in\partial\Omega
\end{align}
that allows us to compute the Neumann boundary values
$\frac{\partial u}{\partial n}$ on $\partial\Omega$ from the
Dirichlet boundary values.
Once we know both, the solution $u$ can be evaluated anywhere in
the domain $\Omega$.

In order to solve the integral equation (\ref{eq:integral}),
we employ a Galerkin discretization:
the unknown Neumann values are approximated by a boundary
element basis $(\varphi_i)_{i\in\Idx}$ and the Dirichlet values by another,
possibly different, basis $(\psi_j)_{j\in\Jdx}$.
The discretization replaces the integral operators by matrices
$G\in\bbbc^{\Idx\times\Idx}$ and $K\in\bbbc^{\Idx\times\Jdx}$ given by
\begin{align*}
  g_{ij} &:= \int_{\partial\Omega} \varphi_i(x)
             \int_{\partial\Omega} g(x,y) \varphi_j(y) \,dy\,dx &
  &\text{ for all } i,j\in\Idx,\\
  k_{ij} &:= \int_{\partial\Omega} \varphi_i(x)
             \int_{\partial\Omega} \frac{\partial g}{\partial n_y}(x,y)
                                   \psi(y) \,dy\,dx &
  &\text{ for all } i\in\Idx,\ j\in\Jdx.
\end{align*}
Both matrices are densely populated in general, and having to store
them explicitly would severely limit the resolution and therefore the
accuracy of the approximation.

Local low-rank approximations offer an attractive solution:
if the kernel function can be approximated by a tensor product,
the corresponding part of the matrix can be approximated by
a low-rank matrix, and keeping this matrix in factorized form
will significantly reduce the storage requirements.

Directional interpolation \cite{BR91,MESCDA12,BOME15,BO22} offers a
particularly convenient approach: the kernel function is split
into a plane wave and a smooth remainder, and interpolation of
the remainder yields a tensor-product approximation of the
kernel function.
Advantages of this approach include ease of implementation and
very robust convergence properties.
A major disadvantage is the large amount of storage required
by this approximation.

Fortunately, this disadvantage can be overcome by combining the
analytical approximation with an algebraic recompression
\cite{BO17,BOBO18} to significantly reduce the storage requirements
and improve the speed of matrix-vector multiplications at the
expense of some additional work.
In order to guarantee the quality of the recompression, the
algorithm relies on weight matrices that describe how ``important''
different basis vectors are for the final approximation.

If we are considering problems with high wave numbers and
high resolutions, these weight matrices may require far more
storage than the final result of the compression, i.e., we may
run out of storage even if the final result would fit a given
computer system.

To solve this problem, we present an algorithm for replacing
the exact weight matrices by compressed weight matrices.
The key challenge is to ensure that the additional errors introduced
by this procedure can be controlled and do not significantly reduce
the accuracy of the final result of the computation.

The following Section~\ref{se:dh2matrix} introduces the structure of
$\mathcal{DH}^2$-matrices used to represent operators
for high-frequency Helmholtz boundary integral equations.
Section~\ref{se:dh2comp} shows how algebraic compression can be
applied to significantly reduce the storage requirements of
$\mathcal{DH}^2$-matrices.
Section~\ref{se:errors} is focused on deriving error estimates
for the compression.
In Section~\ref{se:weights}, we introduce an algorithm for
approximating the weight matrices while preserving the necessary
accuracy.
Section~\ref{se:experiments} contains numerical results indicating
that the new method preserves the convergence rates of the underlying
Galerkin discretization.

\section{\texorpdfstring{$\mathcal{DH}^2$-matrices}
                                           {DH2-matrices}}
\label{se:dh2matrix}

Integral operators with smooth kernel functions can be handled
efficiently by applying interpolation to the kernel function,
since this gives rise to a low-rank approximation.

The kernel function of the high-frequency Helmholtz equation
is oscillatory and therefore not well-suited for interpolation.
The idea of directional interpolation \cite{BR91,MESCDA12,BOME15}
is to split the kernel function into an oscillatory part that can
be approximated by a plane wave and a smooth part that can be
approximated by interpolation.

\subsection{Directional interpolation}

To illustrate this approach, we approximate the oscillatory
part $\exp(\iota \kappa \|x-y\|)$ for $x\in\tau$, $y\in\sigma$,
where $\tau,\sigma\subseteq\bbbr^3$ are star-shaped subsets
with respect to centers $x_\tau\in\tau$ and $y_\sigma\in\sigma$.
A Taylor expansion of $z := x-y$ around $z_0:=x_\tau-y_\sigma$
yields
\begin{align*}
  \kappa \|z\| &= \kappa \|z_0\|
                  + \kappa \langle \frac{z_0}{\|z_0\|}, z - z_0 \rangle\\
        &\quad + \kappa \int_0^1 (1-t)
                \sin^2 \angle( z-z_0, z_0+t(z-z_0) )
                \frac{\|z-z_0\|^2}{\|z_0 + t(z-z_0)\|} \,dt.
\end{align*}
Inserted into the exponential function, the first two terms on the
right-hand side correspond to a plane wave.
In order to ensure that this plane wave is a reasonably good approximation
of the spherical wave appearing in the kernel function, we have to bound
the integral term.
Using the diameter and distance given by
\begin{align*}
  \diam(\tau) &:= \max\{ \|x_1-x_2\|\ :\ x_1,x_2\in\tau \},\\
  \dist(\tau,\sigma) &:= \min\{ \|x-y\|\ :\ x\in\tau,\ y\in\sigma \},
\end{align*}
the third term is bounded if
\begin{subequations}\label{eq:admissibility}
\begin{equation}\label{eq:admissibility_1}
  \kappa \max\{\diam(\tau)^2, \diam(\sigma)^2\} \leq \eta_3 \dist(\tau,\sigma)
\end{equation}
holds with a suitable parameter $\eta_3\in\bbbr_{>0}$.
In terms of our kernel function, this means that we can approximate
the spherical wave $\exp(\iota \kappa \|x-y\|)$ by the plane wave
travelling in direction $z_0$.

Since $z_0$ depends on $x_\tau$ and $y_\sigma$, we would have to use
different directions for every pair $(\tau,\sigma)$ of subdomains,
and this would make the approximation too expensive.
To keep the number of directions under control, we restrict ourselves
to a fixed set $\mathcal{D}$ of unit vectors and approximate $z_0/\|z_0\|$ by
an element $c\in\mathcal{D}$.
If we can ensure
\begin{equation}\label{eq:admissibility_2}
  \kappa \left\| \frac{x_\tau-y_\sigma}{\|x_\tau-y_\sigma\|}
                 - c \right\| \max\{ \diam(\tau), \diam(\sigma) \}
  \leq \eta_2
\end{equation}
with a parameter $\eta_2\in\bbbr_{>0}$, the spherical wave divided
by the plane wave
\begin{equation*}
  \frac{\exp(\iota \kappa \|x-y\|)}{\exp(\iota \kappa \langle c, x-y \rangle)}
  = \exp(\iota\kappa (\|x-y\| - \langle c, x-y \rangle))
\end{equation*}
will still be sufficiently smooth, and the modified kernel function
\begin{align*}
  g_c(x,y) &:= \frac{\exp(\iota \kappa (\|x-y\| - \langle c,x-y \rangle))}
                    {4\pi \|x-y\|} &
  &\text{ for all } x\in\tau,\ y\in\sigma
\end{align*}
will no longer be oscillatory.
In order to interpolate this function, we also have to keep its
denominator under control.
This can be accomplished by requiring
\begin{equation}\label{eq:admissibility_3}
  \max\{\diam(\tau),\diam(\sigma)\} \leq \eta_1 \dist(\tau,\sigma).
\end{equation}
\end{subequations}
If the three \emph{admissibility conditions} (\ref{eq:admissibility_1}),
(\ref{eq:admissibility_2}), and (\ref{eq:admissibility_3}) hold, standard
tensor interpolation of $g_c$ converges at a robust rate
\cite{BOME15,BO22}.

We choose interpolation points $(\xi_{\tau,\nu})_{\nu=1}^k$ with
corresponding Lagrange polynomials $(\ell_{\tau,\nu})_{\nu=1}^k$ in
the subdomain $\tau$ and interpolation points $(\xi_{\sigma,\mu})_{\mu=1}^k$
with corresponding Lagrange polynomials $(\ell_{\sigma,\mu})_{\mu=1}^k$
in the subdomain $\sigma$ and approximate $g_c$ by the interpolating
polynomial
\begin{equation*}
  \tilde g_{\tau\sigma c}(x,y)
  := \sum_{\nu=1}^k \sum_{\mu=1}^k
     g_c(\xi_{\tau,\nu}, \xi_{\sigma,\mu})
     \ell_{\tau,\nu}(x) \overline{\ell_{\sigma,\mu}(y)}.
\end{equation*}
In order to obtain an approximation of the original kernel function
$g$, we have to multiply $g_c$ by the plane wave
$\exp(\iota \kappa \langle c,x-y \rangle)$ and get
\begin{equation*}
  \tilde g_{\tau\sigma}(x,y)
  = \sum_{\nu=1}^k \sum_{\mu=1}^k
    g_c(\xi_{\tau,\nu}, \xi_{\sigma,\mu})
    \ell_{\tau c,\nu}(x) \overline{\ell_{\sigma c,\mu}(y)}
\end{equation*}
with the modified Lagrange functions
\begin{align*}
  \ell_{\tau c,\nu}(x) &= \exp(\iota \kappa \langle c, x \rangle)
                         \ell_{\tau,\nu}(x), &
  \ell_{\sigma c,\mu}(y) &= \exp(\iota \kappa \langle c, y \rangle)
                         \ell_{\sigma,\mu}(y),
\end{align*}
where we exploit
\begin{align*}
  \overline{\ell_{\sigma c,\mu}(y)}
  &= \overline{\exp(\iota \kappa \langle c, y \rangle)}
     \ell_{\sigma,\mu}(y)
   = \exp(-\iota \kappa \langle c, y \rangle)
     \ell_{\sigma,\mu}(y) &
  &\text{ for all } y\in\bbbr^3,\ \mu\in[1:k].
\end{align*}

\subsection{\texorpdfstring{$\mathcal{DH}^2$-matrices}
                           {DH2-matrices}}

To obtain an approximation of the entire matrix $G$, we have to
partition its index set $\Idx\times\Idx$ into subsets where our
approximation can be used.

%
%
\begin{definition}[Cluster tree]
Let $\mathcal{T}$ be a finite tree, and let each of its nodes
$t\in\mathcal{T}$ be associated with a subset $\hat t\subseteq\Idx$.

$\mathcal{T}$ is called a \emph{cluster tree} for the index set
$\Idx$ if
\begin{itemize}
  \item the root $r=\treeroot(\mathcal{T})$ is associated with
     $\hat r = \Idx$,
  \item for all $t\in\mathcal{T}$ with children, we have
     \begin{equation*}
       \hat t = \bigcup_{t'\in\chil(t)} \hat t'.
     \end{equation*}
  \item for all $t\in\mathcal{T}$, $t_1,t_2\in\chil(t)$, we have
     \begin{equation*}
       t_1\neq t_2 \Rightarrow \hat t_1\cap\hat t_2=\emptyset.
     \end{equation*}
\end{itemize}
A cluster tree for $\Idx$ is usually denoted by $\ctI$.
Its leaves are denoted by $\lfI := \{ t\in\ctI\ :\ \chil(t)=\emptyset \}$.
\end{definition}

A cluster tree provides us with a hierarchy of subsets of the index
set $\Idx$, and its leaves define a disjoint partition of $\Idx$.
In order to define approximations for the matrix $G$, we require
a similar tree structure with subsets of $\Idx\times\Idx$.

%
%
\begin{definition}[Block tree]
Let $\mathcal{T}$ be a finite tree.
It is called a \emph{block tree} for the cluster tree $\ctI$ if
\begin{itemize}
  \item for all $b\in\mathcal{T}$ there are $t,s\in\ctI$
     with $b=(t,s)$,
  \item the root $r=\treeroot(\mathcal{T})$ is given by
     $r=(\treeroot(\ctI),\treeroot(\ctI))$,
  \item for all $b=(t,s)\in\mathcal{T}$ with $\chil(b)\neq\emptyset$
     we have
     \begin{equation*}
       \chil(b) = \chil(t)\times\chil(s).
     \end{equation*}
\end{itemize}
A block tree for $\ctI$ is usually denoted by $\ctII$.
Its leaves are denoted by $\lfII := \{ b\in\ctII\ :\ \chil(b)=\emptyset \}$.
\end{definition}

The definition implies that a block tree $\ctII$ for $\ctI$ is indeed
a cluster tree for the index set $\Idx\times\Idx$, and therefore
the leaves $\lfII$ of a block tree describe a disjoint partition of
$\Idx\times\Idx$, i.e., a decomposition of $G$ into submatrices
$G|_{\hat t\times\hat s}$ for all $b=(t,s)\in\lfII$.

We cannot expect to be able to approximate the submatrices intersecting
the diagonal due to the kernel function's singularity, but we can
use the conditions (\ref{eq:admissibility}) to choose those leaves
of $\ctII$ that can be approximated.

In order to be able to apply (\ref{eq:admissibility}), we need to
take the supports of the basis functions into account.
Since we will be using tensor interpolation, we choose for every
cluster $t\in\ctI$ an axis-parallel \emph{bounding box}
$\tau\subseteq\bbbr^3$ such that
\begin{align*}
  \supp\varphi_i &\subseteq \tau &
  &\text{ for all } i\in\hat t.
\end{align*}
For every cluster $s\in\ctI$, we denote the corresponding bounding box
by $\sigma$.
If we have a block $(t,s)\in\lfII$ with bounding boxes $\tau$ and
$\sigma$ satisfying the \emph{admissibility conditions}
(\ref{eq:admissibility}), we can expect the approximation
\begin{align*}
  \tilde g_{\tau\sigma}(x,y)
  &= \sum_{\nu=1}^k \sum_{\mu=1}^k
       g_c(\xi_{\tau,\nu}, \xi_{\sigma,\mu})
       \ell_{\tau c,\nu}(x) \overline{\ell_{\sigma c,\mu}(y)} &
  &\text{ for all } x\in\tau,\ y\in\sigma
\end{align*}
for a suitably chosen direction $c$ to converge rapidly and therefore
\begin{align}
  g_{ij}
  &\approx \int_{\partial\Omega} \varphi_i(x)
     \int_{\partial\Omega} \tilde g_{\tau\sigma}(x,y)
             \varphi_j(y) \,dy\,dx\notag\\
  &= \sum_{\nu=1}^k \sum_{\mu=1}^k
       \underbrace{\int_{\partial\Omega} \ell_{\tau c,\nu}(x) \varphi_i(x) \,dx
                  }_{=:v_{tc,i\nu}}
       \underbrace{g_c(\xi_{\tau,\nu}, \xi_{\sigma,\mu})
                  }_{=:s_{ts,\nu\mu}}
       \underbrace{\int_{\partial\Omega} \overline{\ell_{\sigma c,\mu}(y) \varphi_j(y)} \,dy
                  }_{=:\bar v_{sc,j\mu}}\label{eq:vsw_inter}\\
  &= (V_{tc} S_{ts} V_{sc}^*)_{ij}
      \qquad\text{ for all } i\in\hat t,\ j\in\hat s.\notag
\end{align}
This means that the submatrices corresponding to the leaves
\begin{equation*}
  \lfaII := \{ (t,s)\in\lfII\ :\ \tau \text{ and } \sigma
                 \text{ satisfy (\ref{eq:admissibility})} \}
\end{equation*}
can be approximated by low-rank matrices in factorized form.

We can satisfy the admissibility condition (\ref{eq:admissibility_2})
only if large clusters are accompanied by a large number of directions
to choose from.

%
%
\begin{definition}[Directions]
Let $\ctI$ be a cluster tree.
For every cluster $t\in\ctI$ we let either $\mathcal{D}_t=\{0\}$
or choose a subset $\mathcal{D}_t\subseteq\bbbr^3$ such that
\begin{align*}
  \|c\| &= 1 &
  &\text{ for every } c\in\mathcal{D}_t.
\end{align*}
The family $(\mathcal{D}_t)_{t\in\ctI}$ is called a \emph{family
of directions} for the cluster tree $\ctI$.
\end{definition}

Allowing $\mathcal{D}_t=\{0\}$ makes algorithms more efficient
for small clusters where (\ref{eq:admissibility_2}) can be fulfilled
by choosing $c=0$.
In this case, the function $\ell_{\tau c,\nu}$ becomes simply the
Lagrange polynomial $\ell_{\tau,\nu}$, and the modified kernel function $g_c$
becomes just the standard kernel function $g$.

Storing the matrices $(V_{tc})_{t\in\ctI,c\in\mathcal{D}_t}$ for all
clusters and all directions would generally require $\mathcal{O}(n^2)$
coefficients, where $n=\#\Idx$ denotes the number of basis functions,
and this would not be an improvement over simply storing the matrix
explicitly.
This problem can be overcome by taking advantage of the fact that we
can approximate $V_{tc}$ in terms of the matrices $V_{t'c'}$ corresponding
to its children:
if we use the same polynomial order for all clusters, we have
\begin{align*}
  \ell_{\tau,\nu}
  &= \sum_{\nu'=1}^k \ell_{\tau,\nu}(\xi_{\tau',\nu'}) \ell_{\tau',\nu'} &
  &\text{ for all } \nu\in[1:k]
\end{align*}
by the identity theorem, and interpolating a slightly modified function
instead yields
\begin{align*}
  \ell_{\tau c,\nu}(x)
  &= \exp(\iota \kappa \langle c, x \rangle) \ell_{\tau,\nu}(x)\\
  &= \exp(\iota \kappa \langle c', x \rangle)
     \exp(\iota \kappa \langle c-c', x \rangle) \ell_{\tau,\nu}(x)\\
  &\approx \exp(\iota \kappa \langle c', x \rangle)
     \sum_{\nu'=1}^k
       \underbrace{\exp(\iota \kappa \langle c-c', \xi_{\tau',\nu'} \rangle)
                    \ell_{\tau,\nu}(\xi_{\tau',\nu'})}_{=:e_{\tau' c,\nu'\nu}}
       \ell_{\tau',\nu'}(x)\\
  &= \sum_{\nu'=1}^k e_{\tau' c,\nu'\nu}\, \ell_{\tau' c',\nu'}(x),
\end{align*}
i.e., we can approximate modified Lagrange polynomials on parent clusters
by modified Lagrange polynomials in their children.
Under moderate conditions, this approximation can be applied repeatedly
without harming the total error too much \cite{BOME15,BO22}, so we can
afford to replace the matrices $V_{\tau c}$ defined in (\ref{eq:vsw_inter}) in
all non-leaf clusters by approximations.

%
%
\begin{definition}[Directional cluster basis]
\label{de:cluster_basis}
Let $\ctI$ be a cluster tree with a family $\mathcal{D}=(\mathcal{D}_t)_{t\in\ctI}$
of directions.
A family $(V_{tc})_{t\in\ctI, c\in\mathcal{D}_t}$ of matrices
$V_{tc}\in\bbbc^{\hat t\times k}$ is called a
\emph{directional cluster basis} for $\ctI$ and $\mathcal{D}$ if
for every $t\in\ctI$ and $t'\in\chil(t)$ there are a direction
$c'=\dirchil(t',c)$ and a matrix $E_{t' c}\in\bbbc^{k\times k}$ with
\begin{equation}\label{eq:transfer}
  V_{tc}|_{\hat t'\times k} = V_{t'c'} E_{t' c}.
\end{equation}
The matrices $E_{t' c}$ are called \emph{transfer matrices}.
Since the matrices $V_{tc}$ have to be stored only for clusters without
children, they are called \emph{leaf matrices}.
\end{definition}

%
%
\begin{definition}[$\mathcal{DH}^2$-matrix]
Let $\ctI$ be a cluster tree with a family $\mathcal{D}$ of directions,
let $V=(V_{tc})_{t\in\ctI, c\in\mathcal{D}_t}$ be a directional cluster basis,
and let $\ctII$ be a block tree.

A matrix $G\in\bbbc^{\Idx\times\Idx}$ is called a
\emph{$\mathcal{DH}^2$-matrix} if for every admissible leaf
$b=(t,s)\in\lfaII$ there are a direction
$c=\dirblock(t,s)\in\mathcal{D}_t\cap\mathcal{D}_s$
and a matrix $S_{ts}\in\bbbc^{k\times k}$ with
\begin{equation}\label{eq:vsw}
  G|_{\hat t\times\hat s} = V_{tc} S_{ts} V_{sc}^*.
\end{equation}
The matrix $S_{ts}$ is called a \emph{coupling matrix} for the block
$b=(t,s)$.
\end{definition}

%
%
\begin{figure}
  \begin{center}
    \includegraphics[width=0.45\textwidth]{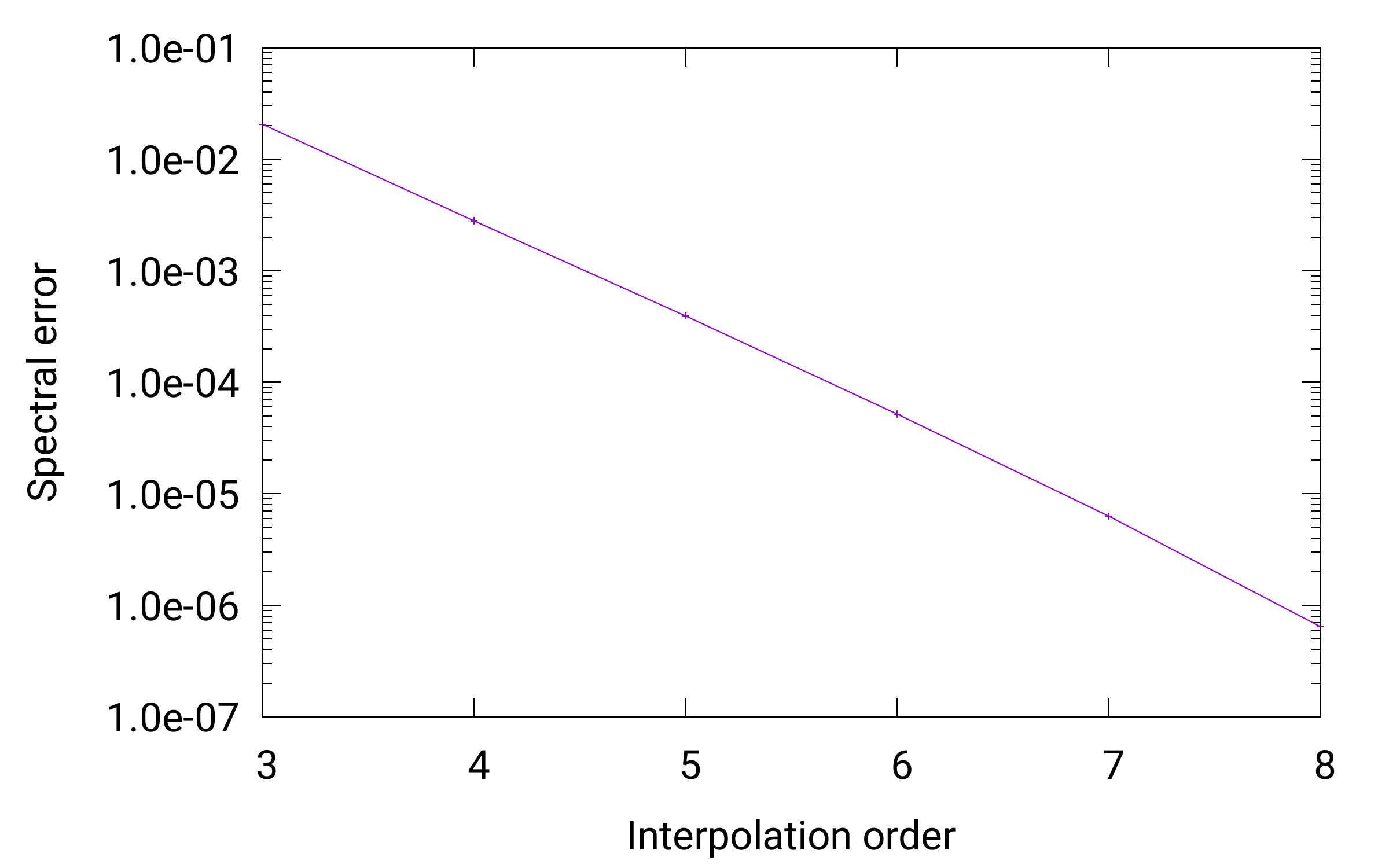}%
    \qquad%
    \includegraphics[width=0.45\textwidth]{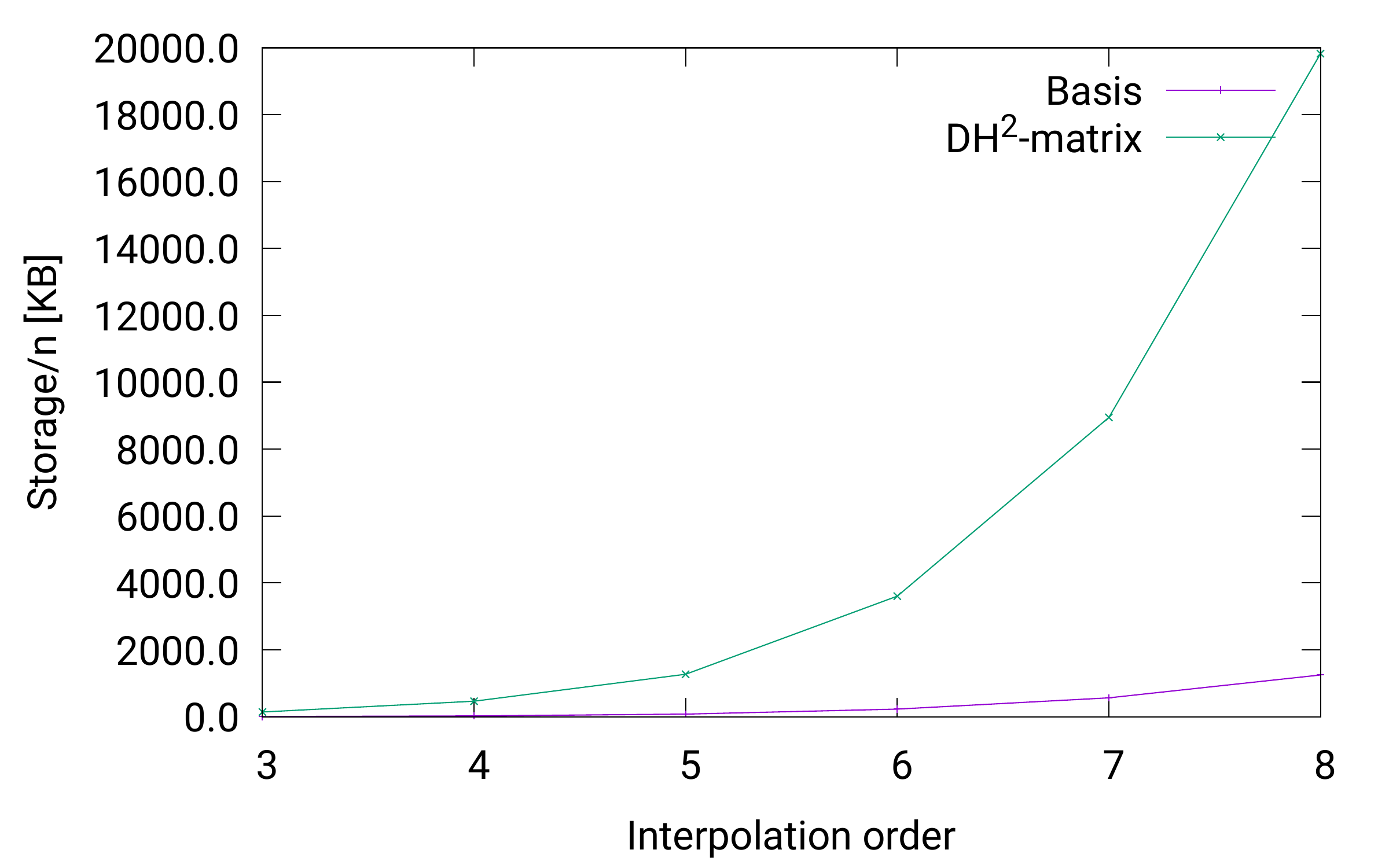}
  \end{center}
 
  \caption{Left: Convergence of directional interpolation with increasing
     order.
     Right: Storage requirements with increasing order}
  \label{fi:directional}
\end{figure}

If we have a $\mathcal{DH}^2$-matrix, all admissible blocks
can be represented by the cluster basis and the coupling matrices.
For the inadmissible leaves $b=(t,s)\in\lfiII$, we store the
corresponding submatrices $G|_{\hat t\times\hat s}$ explicitly.
Under moderate assumptions, these \emph{nearfield matrices}
require only $\mathcal{O}(n k)$ units of storage.

For constant wave numbers $\kappa$, a $\mathcal{DH}^2$-matrix
approximation of $G$ requires only $\mathcal{O}(n k^2)$ units of
storage.
In the high-frequency case, i.e., if $\kappa \sim \sqrt{n}$,
a $\mathcal{DH}^2$-matrix approximation requires only
$\mathcal{O}(n k^3 \log n)$ units of storage \cite{BOME15,BO17}.
The matrix-vector multiplication can be performed in a similar
complexity.

\section{\texorpdfstring{Compression of $\mathcal{DH}^2$-matrices}
                                           {Compression of DH2-matrices}}
\label{se:dh2comp}

Although directional interpolation leads to a robust and fairly
fast algorithm for approximating the matrix $G$, it requires
a very large amount of storage, particularly if we are interested
in a highly accurate approximation:
Figure~\ref{fi:directional} shows that directional interpolation
converges very robustly, but also that interpolation of higher
order requires a very large amount of memory, close to 1~TB
for the eigth order.

Algebraic compression techniques offer an attractive solution:
we use directional interpolation only to provide an intermediate
approximation of $G$ and apply algebraic techniques to reduce
the rank as far as possible.
The resulting re-compression algorithm can be implemented in a
way that avoids having to store the entire intermediate
$\mathcal{DH}^2$-matrix, so that only the final result is
written to memory and very large matrices can be handled at
high accuracies.

We present here a version of the algorithm introduced in
\cite{BOBO18} that will be modified in the following sections.
Our goal is to find an improved cluster basis
$Q=(Q_{tc})_{t\in\ctI,c\in\mathcal{D}_t}$ for the matrix $G$.
In order to avoid redundant information and to ensure numerical
stability, we aim for an \emph{isometric} basis, i.e., we require
\begin{align*}
  Q_{tc}^* Q_{tc} &= I &
  &\text{ for all } t\in\ctI,\ c\in\mathcal{D}_t.
\end{align*}
The best approximation of a matrix block $G|_{\hat t\times\hat s}$
with respect to this basis is given by the orthogonal projection
$Q_{tc} Q_{tc}^* G|_{\hat t\times\hat s}$, and we have to ensure
that all blocks connected to the cluster $t$ and the direction
$c$ are approximated.
We introduce the sets
\begin{align}\label{eq:Rtc_definition}
  \mathcal{R}_{tc} &:= \{ s\in\ctI\ :\ (t,s)\in\lfaII,
                             \ \dirblock(t,s)=c \} &
  &\text{ for all } t\in\ctI,\ c\in\mathcal{D}_t
\end{align}
containing all column clusters connected via an admissible block
to a row cluster $t$ and a given direction $c$ and note that
\begin{align*}
  G|_{\hat t\times\hat s} &\approx Q_{tc} Q_{tc}^* G|_{\hat t\times\hat s} &
  &\text{ for all } s\in\mathcal{R}_{tc}
\end{align*}
is a minimal requirement for our new basis.
But it is not entirely sufficient:
if $t\in\ctI$ has children, Definition~\ref{de:cluster_basis} requires
that $Q_{tc}|_{\hat t'\times k}$ can be expressed in terms of $Q_{t'c'}$
for $t'\in\chil(t)$ and $c'=\dirchil(t',c)$, therefore the basis
$Q_{tc}$ has to be able to approximate all admissible blocks connected
to the ancestors of $t$, as well.
To reflect this requirement, we extend $\mathcal{R}_{tc}$ to
\begin{align}\label{eq:admissible_row}
  \mathcal{R}_{tc}^*
  &:= \begin{cases}
     \mathcal{R}_{tc} &\text{ if } t \text{ is the root of } \ctI,\\
     \mathcal{R}_{tc} \cup
       \bigcup_{\substack{c^+\in\mathcal{D}_{t^+}\\\dirchil(t,c^+)=c}}
       \mathcal{R}_{t^+ c^+}^* &\text{ if } t\in\chil(t^+),\ t^+\in\ctI
  \end{cases} &
  &\text{ for all } t\in\ctI,\ c\in\mathcal{D}_t
\end{align}
by including all admissible blocks connected to the parent, and by
induction to any of its ancestors.
A suitable cluster basis satisfies
\begin{align*}
  G|_{\hat t\times\hat s} &\approx Q_{tc} Q_{tc}^* G|_{\hat t\times\hat s} &
  &\text{ for all } s\in\mathcal{R}_{tc}^*,\ t\in\ctI,\ c\in\mathcal{D}_t.
\end{align*}
By combining all of these submatrices in a large matrix
\begin{align*}
  G_{tc} &:= G|_{\hat t\times\mathcal{C}_{tc}}, &
  \mathcal{C}_{tc} &:= \bigcup\{ \hat s\ :\ s\in\mathcal{R}_{tc}^* \},
\end{align*}
we obtain the equivalent formulation
\begin{align*}
  G_{tc} &\approx Q_{tc} Q_{tc}^* G_{tc} &
  &\text{ for all } t\in\ctI,\ c\in\mathcal{D}_t,
\end{align*}
and the singular value decompositions of $G_{tc}$ can be used to determine
optimal isometric matrices $Q_{tc}$ with this property.
The resulting algorithm, investigated in \cite{BO17}, has quadratic
complexity, since it does not take the special structure of $G$ into
account.

If $G$ is already approximated by a $\mathcal{DH}^2$-matrix, e.g.,
via directional interpolation, we can make this algorithm considerably
more efficient.
We start by considering the root $t$ of $\ctI$.
Let $c\in\mathcal{D}_t$.
Since $G$ is a $\mathcal{DH}^2$-matrix, we have
\begin{align*}
  G|_{\hat t\times\hat s} &= V_{tc} S_{ts} V_{sc}^* &
  &\text{ for all } s\in\mathcal{R}_{tc},
\end{align*}
and enumerating $\mathcal{R}_{tc} = \{ s_1, \ldots, s_m \}$ yields
\begin{equation*}
  G_{tc} = V_{tc} \begin{pmatrix}
     S_{ts_1} V_{s_1c}^* & \ldots & S_{ts_m} V_{s_mc}^*
  \end{pmatrix}.
\end{equation*}
The right factor has only $k$ rows, and we can use Householder
transformations to condense it into a small $k\times k$ matrix
without changing the singular values and left singular vectors of
$G_{tc}$.
Using the transformations directly, however, is too computationally
expensive, so we are looking for way to avoid it.

%
%
\begin{definition}[Basis weights]
\label{de:basis_weights}
A family $(R_{sc})_{s\in\ctI, c\in\mathcal{D}_s}$ of matrices is called
a family of \emph{basis weights} for the basis $(V_{sc})_{s\in\ctI,\
  c\in\mathcal{D}_s}$ if for every $s\in\ctI$ and $c\in\mathcal{D}_s$
there is an isometric matrix $Q_{sc}$ with
\begin{equation*}
  V_{sc} = Q_{sc} R_{sc}
\end{equation*}
and the matrices $R_{sc}$ have each $k$ columns and at most $k$ rows.
\end{definition}

If we have basis weights at our disposal, we obtain
\begin{equation*}
   G_{tc} = V_{tc}
     \begin{pmatrix}
       S_{ts_1} R_{s_1c}^* & \ldots & S_{ts_m} R_{s_mc}^*
     \end{pmatrix}
     \begin{pmatrix}
       Q_{s_1c}^* & & \\
       & \ddots & \\
       & & Q_{s_mc}^*
     \end{pmatrix},
\end{equation*}
and since the multiplication by an adjoint isometric matrix from the
left does not change the singular values or left singular vectors,
we can replace $G_{tc}$ with
\begin{equation*}
   V_{tc} \begin{pmatrix}
            S_{ts_1} R_{s_1c}^* & \ldots & S_{ts_m} R_{s_mc}^*
          \end{pmatrix}.
\end{equation*}
We can even go one step further and compute a thin Householder
factorization of the right factor's adjoint
\begin{equation*}
  \widehat{P}_{tc} Z_{tc}
  = \begin{pmatrix}
       R_{s_1c} S_{ts_1}^*\\
       \vdots\\
        R_{s_mc} S_{ts_m}^*
     \end{pmatrix}
\end{equation*}
with an isometric matrix $\widehat{P}_{tc}$ and a matrix $Z_{tc}$ that has
only $k$ columns and not more than $k$ rows.
If we set
\begin{equation*}
  P_{tc} := \begin{pmatrix}
    Q_{s_1c} & & \\
    & \ddots & \\
    & & Q_{s_mc}
  \end{pmatrix} \widehat{P}_{tc},
\end{equation*}
we obtain
\begin{equation*}
  G_{tc} = V_{tc} Z_{tc}^* P_{tc}^*
\end{equation*}
and can drop the rightmost adjoint isometric matrix to work just with the thin
matrix $V_{tc} Z_{tc}^*$ that has only at most $k$ columns.

So far, we have only considered the root of the cluster tree.
If $t\in\ctI$ is a non-root cluster, it has a parent $t^+\in\ctI$
and our definition (\ref{eq:admissible_row}) yields
\begin{equation*}
  \mathcal{R}_{tc}^* = \mathcal{R}_{tc} \cup
       \bigcup_{\substack{c^+\in\mathcal{D}_{t^+}\\ \dirchil(t,c^+)=c}}
       \mathcal{R}_{t^+c^+}^*.
\end{equation*}
Let $c^+\in\mathcal{D}_{t^+}$ with $\dirchil(t,c^+)=c$.
If we assume that $Z_{t^+ c^+}$ has already been computed, we have
\begin{equation*}
  G_{tc}|_{\hat t\times\mathcal{C}_{t^+ c^+}}
  = (G_{t^+ c^+})|_{\hat t\times\mathcal{C}_{t^+ c^+}}
  = (V_{t^+ c^+} Z_{t^+ c^+}^* P_{t^+ c^+}^*
    )|_{\hat t\times\mathcal{C}_{t^+ c^+}}
  = V_{tc} E_{t c^+} Z_{t^+ c^+}^* P_{t^+ c^+}^*.
\end{equation*}
To apply this procedure to all directions $c^+$, we enumerate them as
\begin{equation*}
  \{ c_1^+, \ldots, c_\ell^+ \}
  = \{ c^+\in\mathcal{D}_{t^+}\ :\ \dirchil(t,c^+)=c \}
\end{equation*}
and the admissible blocks again as $\mathcal{R}_{tc} = \{ s_1, \ldots, s_m \}$
to get
\begin{align*}
  G_{tc}
  &= V_{tc} \begin{pmatrix}
      S_{ts_1} V_{s_1c}^* & \ldots & S_{ts_m} V_{s_mc}^* &
      E_{tc^+_1} Z_{t^+c^+_1}^* P_{t^+c^+_1}^* & \ldots
      & E_{tc^+_\ell} Z_{t^+c^+_\ell}^* P_{t^+c^+_\ell}^*
    \end{pmatrix}\\
  &= V_{tc} \begin{pmatrix}
      S_{ts_1} R_{s_1c}^* Q_{s_1c}^* & \ldots & S_{ts_m} R_{s_mc}^* Q_{S_mc}^* &
      E_{tc^+_1} Z_{t^+c^+_1}^* P_{t^+c^+_1}^* & \ldots
      & E_{tc^+_\ell} Z_{t^+c^+_\ell}^* P_{t^+c^+_\ell}^*
    \end{pmatrix}.
\end{align*}
The rightmost factors are again isometric, and we can once more compute a
thin Householder factorization
\begin{equation*}
  \widehat{P}_{tc} Z_{tc}
  = \begin{pmatrix}
       R_{s_1c} S_{ts_1}^* \\
       \vdots \\
       R_{s_mc} S_{ts_m}^* \\
       Z_{t^+c^+_1} E_{tc^+_1}^* \\
       \vdots \\
       Z_{t^+c^+_\ell} E_{tc^+_\ell}^*
    \end{pmatrix}
  \quad\text{and set}\quad
  P_{tc} := \begin{pmatrix}
      Q_{s_1c} & & & & & \\
      & \ddots & & & & \\
      & & Q_{s_mc} & & & \\
      & & & P_{t^+c^+_1} & & \\
      & & & & \ddots & \\
      & & & & & P_{t^+c^+_\ell}
    \end{pmatrix} \widehat{P}_{tc}.
\end{equation*}
to obtain
\begin{equation*}
  G_{tc} = V_{tc} Z_{tc}^* P_{tc}^*.
\end{equation*}
Since the isometric matrices $P_{tc}$ do not influence the range
of $G_{tc}$, we do not have to compute them, we only need the
weight matrices $Z_{tc}$.

%
%
\begin{definition}[Total weights]
A family $(Z_{tc})_{t\in\ctI, c\in\mathcal{D}_t}$ of matrices is called
a family of \emph{total weights} for the $\mathcal{DH}^2$-matrix $G$
if for every $t\in\ctI$ and $c\in\mathcal{D}_t$ there is an
isometric matrix $P_{tc}$ with
\begin{equation}\label{eq:Gtc_factorization}
  G_{tc} = V_{tc} Z_{tc}^* P_{tc}^*
\end{equation}
and the matrices $Z_{tc}$ have each $k$ columns and at most $k$ rows.
\end{definition}

%
%
\begin{remark}[Symmetric total weights]
In the original approximation constructed by directional interpolation,
the same cluster basis is used for rows and columns, since we have
$G|_{\hat t\times\hat s} = V_{tc} S_{ts} V_{sc}^*$ for all admissible
blocks $b=(t,s)\in\lfaII$.

Since the matrix $G$ is not symmetric, this property no longer holds
for the adaptively constructed basis $(Q_{tc})_{t\in\ctI,c\in\mathcal{D}_t}$
and we would have to construct a separate basis for the columns by applying
the procedure to the adjoint matrix $G^*$.

A possible alternative is to extend the total weight matrices to handle
$G$ and $G^*$ simultaneously: for every $s\in\mathcal{R}_{tc}$, we
include not only $R_{sc} S_{ts}^*$ in the construction of the weight
$Z_{tc}$, but also $R_{sc} S_{st}$.
This will give us an adaptive cluster basis that can be used for rows and
columns, just like the original.
Since the matrices appearing in the Householder factorization are now
twice as large, the algorithm will take almost twice as long to complete
and the adaptively chosen ranks may increase.
\end{remark}

%
%
\begin{figure}
  \begin{quotation}
    \begin{tabbing}
      \textbf{procedure} basis\_weights($s$);\\
      \textbf{begin}\\
      \quad\= \textbf{if} $\chil(s)=\emptyset$ \textbf{then}\\
      \> \quad\= \textbf{for} $c\in\mathcal{D}_s$ \textbf{do}\\
      \> \> \quad\= Find a thin Householder decomposition
                        $V_{sc} = Q_{sc} R_{sc}$\\
      \> \textbf{else begin}\\
      \> \> \textbf{for} $s'\in\chil(s)$ \textbf{do}\\
      \> \> \> basis\_weights($s'$);\\
      \> \> \textbf{for} $c\in\mathcal{D}_s$ \textbf{do begin}\\
      \> \> \> Set up $\widehat{V}_{sc}$ as in (\ref{eq:Vhat_sc});\\
      \> \> \> Find a thin Householder decomposition
                        $\widehat{V}_{sc} = \widehat{Q}_{sc} R_{sc}$\\
      \> \> \textbf{end}\\
      \> \textbf{end}\\
      \textbf{end}
    \end{tabbing}
  \end{quotation}
  \caption{Construction of the basis weights $R_{sc}$}
  \label{fi:basis_weights}
\end{figure}

We can compute the total weights efficiently by this procedure
as long as we have the basis weights $(R_{sc})_{s\in\ctI,c\in\mathcal{D}_s}$
at our disposal.
These weights can be computed efficiently by taking advantage of
their nested structure:
if $s\in\ctI$ is a leaf, we compute the thin Householder factorization
\begin{equation*}
  V_{sc} = Q_{sc} R_{sc}
\end{equation*}
with an isometric matrix $Q_{sc}$ and a matrix $R_{sc}$ with
$k$ columns and at most $k$ rows.

If $s\in\ctI$ has children, we first compute the basis weights
for all children $\chil(s)=\{s_1,\ldots,s_\ell\}$ by recursion and let
\begin{equation}\label{eq:Vhat_sc}
  \widehat{V}_{sc} = \begin{pmatrix}
    R_{s_1 c_1} E_{s_1 c}\\
    \vdots\\
    R_{s_\ell c_\ell} E_{s_\ell c}
  \end{pmatrix},
\end{equation}
where $c_i = \dirchil(s_i,c)$ for all $i\in[1:\ell]$.
We compute the thin Householder factorization
\begin{equation*}
  \widehat{V}_{sc} = \widehat{Q}_{sc} R_{sc}
\end{equation*}
and find
\begin{equation*}
  V_{sc} = \begin{pmatrix}
             V_{s_1 c_1} E_{s_1 c}\\
             \vdots\\
             V_{s_\ell c_\ell} E_{s_\ell c}
           \end{pmatrix}
  = \begin{pmatrix}
      Q_{s_1 c_1} & & \\
      & \ddots & \\
      & & Q_{s_\ell c_\ell}
    \end{pmatrix} \widehat{V}_{sc}
  = \underbrace{\begin{pmatrix}
    Q_{s_1 c_1} & & \\
    & \ddots & \\
    & & Q_{s_\ell c_\ell}
  \end{pmatrix} \widehat{Q}_{sc}}_{=:Q_{sc}} R_{sc}.
\end{equation*}
The matrix $Q_{sc}$ is the product of two isometric matrices and
therefore itself isometric.
We can see that we can compute the basis weight matrices
$R_{sc}$ using only $\mathcal{O}(k^3)$ operations per $s\in\ctI$
and $c\in\mathcal{D}_s$ as long as we are not interested in $Q_{sc}$.
The algorithm is summarized in Figure~\ref{fi:basis_weights}.

Once the basis weights and total weights have been computed, we
can construct the improved cluster basis $Q_{tc}$.

%
%
\begin{figure}
  \begin{quotation}
    \begin{tabbing}
      \textbf{procedure} build\_basis($t$);\\
      \textbf{begin}\\
      \quad\= \textbf{if} $\chil(t)=\emptyset$ \textbf{then}\\
      \> \quad\= \textbf{for} $c\in\mathcal{D}_s$ \textbf{do begin}\\
      \> \> \quad\= Use a thin Householder factorization to get $Z_{tc}$;\\
      \> \> \> Compute the singular value decomposition
                     $V_{tc} Z_{tc}^* = U \Sigma V^*$;\\
      \> \> \> Choose a rank $k_{tc}$, shrink $U$ to its first $k_{tc}$ columns;\\
      \> \> \> $Q_{tc} \gets U$; \quad $T_{tc} \gets Q_{tc}^* V_{tc}$\\
      \> \> \textbf{end}\\
      \> \textbf{else begin}\\
      \> \> \textbf{for} $s'\in\chil(s)$ \textbf{do}\\
      \> \> \> build\_basis($s'$);\\
      \> \> \textbf{for} $c\in\mathcal{D}_s$ \textbf{do begin}\\
      \> \> \> Use a thin Householder factorization to get $Z_{tc}$;\\
      \> \> \> Set up $\widehat{V}_{tc}$ as in (\ref{eq:Vhat_tc});\\
      \> \> \> Compute the singular value decomposition
                      $\widehat{V}_{tc} Z_{tc}^* = \widehat{U} \Sigma V^*$;\\
      \> \> \> Choose a rank $k_{tc}$, shrink $\widehat{U}$ to its first
                 $k_{tc}$ columns;\\
      \> \> \> $\widehat{Q}_{tc} \gets \widehat{U}$;
               \quad $T_{tc} \gets \widehat{Q}_{rc}^* \widehat{V}_{tc}$\\
      \> \> \textbf{end}\\
      \> \textbf{end}\\
      \textbf{end}
    \end{tabbing}
  \end{quotation}
  \caption{Construction of an adaptive cluster basis}
  \label{fi:build_basis}
\end{figure}

If $t\in\ctI$ is a leaf, we make use of (\ref{eq:Gtc_factorization})
to get
\begin{equation*}
  G_{tc} = V_{tc} Z_{tc} P_{tc}^*,
\end{equation*}
and we can again drop the isometric matrix $P_{tc}$ and only have to
find the singular value decomposition of $V_{tc} Z_{tc}^*$, choose a
suitable rank $k_{tc}\in\bbbn_0$ and use the first $k_{tc}$ left singular
vectors as columns of the matrix $Q_{tc}$.
We also prepare the matrix $T_{tc} := Q_{tc}^* V_{tc}$ describing the
change of basis from $V_{tc}$ to $Q_{tc}$.

If $t\in\ctI$ is not a leaf, we first construct the basis for
all children $\{t_1,\ldots,t_\ell\}=\chil(t)$.
Since the parent can only approximate what has been kept by its children,
we can switch to the orthogonal projection
\begin{equation*}
  \widehat{G}_{tc} := \begin{pmatrix}
    Q_{t_1 c_1}^* & & \\
    & \ddots & \\
    & & Q_{t_\ell c_\ell}^*
  \end{pmatrix} G_{tc}
  = \begin{pmatrix}
    Q_{t_1 c_1}^* G|_{\hat t_1\times\mathcal{C}_{tc}}\\
    \vdots\\
    Q_{t_\ell c_\ell}^* G|_{\hat t_\ell\times\mathcal{C}_{tc}}
  \end{pmatrix}
\end{equation*}
of $G_{tc}$ with  $c_i=\dirchil(t_i,c)$ for all $i\in[1:\ell]$.
Using again (\ref{eq:Gtc_factorization}), we find
\begin{equation*}
  \widehat{G}_{tc} = \widehat{V}_{tc} Z_{tc} P_{tc}^*,
\end{equation*}
with the projection of $V_{tc}$ into the children's bases
\begin{equation}\label{eq:Vhat_tc}
  \widehat{V}_{tc} := \begin{pmatrix}
    T_{t_1,c_1} E_{t_1 c}\\
    \vdots\\
    T_{t_\ell,c_\ell} E_{t_\ell c}
  \end{pmatrix}
\end{equation}
that can be easily computed using the transfer matrices and the
basis-change matrices.
Once again we can drop the isometric factor $P_{tc}$ and only have
to compute the singular value decomposition of $\widehat{V}_{tc} Z_{tc}^*$,
choose again a suitable rank $k_{tc}\in\bbbn_0$ and use the first
$k_{tc}$ left singular vectors as columns of a matrix $\widehat{Q}_{tc}$.
Using
\begin{equation*}
  Q_{tc} := \begin{pmatrix}
               Q_{t_1,c_1} & & \\
               & \ddots & \\
               & & Q_{t_\ell,c_\ell}
            \end{pmatrix} \widehat{Q}_{tc}
\end{equation*}
gives us the new cluster basis, where the transfer matrices can be
extracted from $\widehat{Q}_{tc}$.
Again it is a good idea to prepare the basis-change matrix
$T_{tc} := Q_{tc}^* V_{tc} = \widehat{Q}_{tc}^* \widehat{V}_{tc}$
for the next steps of the recursion.

Under standard assumptions, the entire construction can be
performed in $\mathcal{O}(n k^2)$ operations for constant wave numbers
and $\mathcal{O}(n k^3 \log n)$ operations in the high-frequency case
\cite{BOBO18}.
The algorithm is summarized in Figure~\ref{fi:build_basis}.
It is important to note that the total weight matrices $Z_{tc}$
can be constructed and discarded during the recursive algorithm,
they do not have to be kept in storage permanently.
This is in contrast to the basis weight matrices $R_{sc}$ that
may appear at any time during the recursion and therefore are
kept in storage during the entire run of the algorithm.

%
%
\begin{figure}
  \begin{center}
    \includegraphics[width=0.45\textwidth]{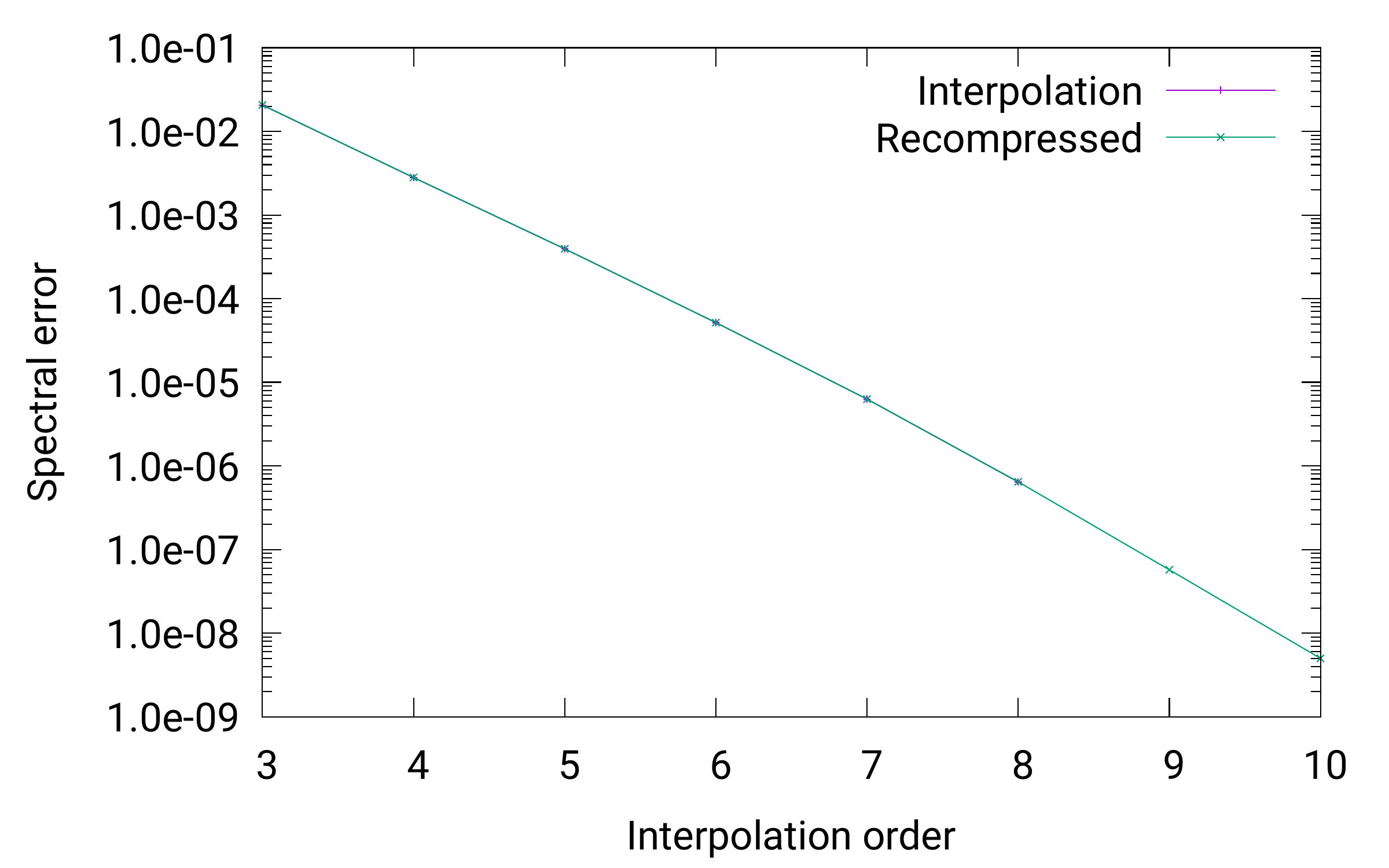}%
    \qquad%
    \includegraphics[width=0.45\textwidth]{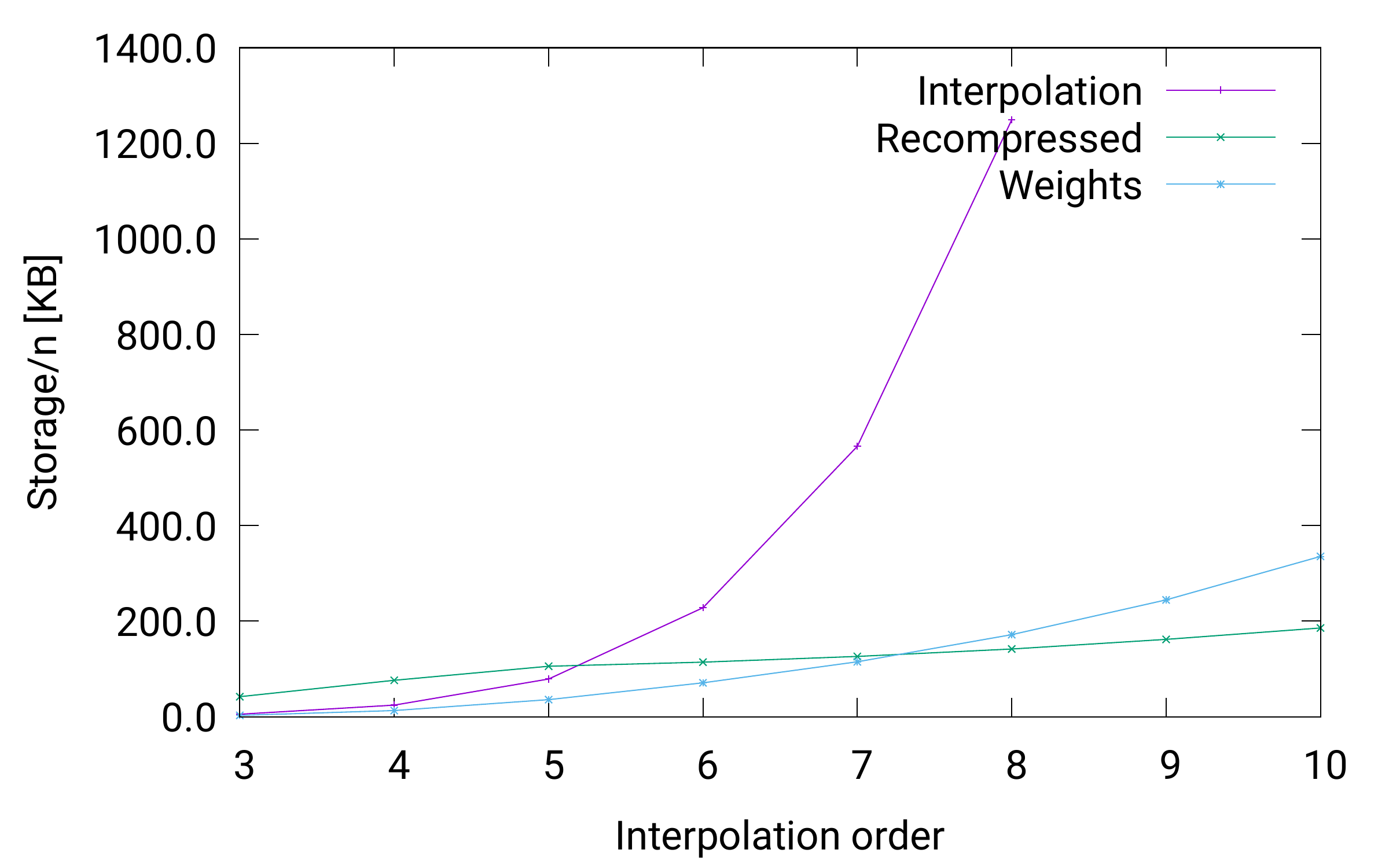}
  \end{center}
 
  \caption{Left: Convergence of recompressed interpolation with
    increasing order.
    Right: Storage requirements of original interpolation and
    recompression}
  \label{fi:recomp}
\end{figure}

Figure~\ref{fi:recomp} shows that recompression --- applied with
suitable parameters --- leaves the approximation quality unchanged
and drastically reduces the storage requirements.
The blue curve corresponds to the storage needed for the basis
weights, and we can see that it grows beyond the storage for the
entire recompressed $\mathcal{DH}^2$-matrix if higher polynomial
orders are used.
This is not acceptable if we want to apply the method to high-frequency
problems at high accuracies, so we will now work to reduce the
storage requirements for these weights without harming the
convergence of the overall method.

\section{Error control}
\label{se:errors}

In order to preserve the convergence properties, we have to investigate
how our algorithm reacts to perturbations.

\subsection{Error decomposition}
We consider an admissible block $(t,s)\in\lfaII$ with $c=\dirblock(t,s)$
that is approximated by our algorithm by
\begin{equation*}
  Q_{tc} Q_{tc}^* G|_{\hat t\times\hat s}.
\end{equation*}
If $t$ is a leaf, the approximation error is given by
\begin{equation*}
  G|_{\hat t\times\hat s} - Q_{tc} Q_{tc}^* G|_{\hat t\times\hat s}.
\end{equation*}
If $t$ is not a leaf, there are children $\{t_1,\ldots,t_\ell\}=\chil(t)$
with directions $c_i=\dirchil(t_i,c)$, $i\in[1:\ell]$, and an isometric
matrix $\widehat{Q}_{tc}$ such that
\begin{equation*}
  Q_{tc} = \underbrace{\begin{pmatrix}
    Q_{t_1 c_1} & & \\
    & \ddots & \\
    & & Q_{t_i c_i}
  \end{pmatrix}}_{=:U_{tc}} \widehat{Q}_{tc} = U_{tc} \widehat{Q}_{tc}
\end{equation*}
and the approximation error can be split into
\begin{align*}
  G|_{\hat t\times\hat s} &- Q_{tc} Q_{tc}^* G|_{\hat t\times\hat s}
   = G|_{\hat t\times\hat s} - U_{tc} \widehat{Q}_{tc} \widehat{Q}_{tc}^* U_{tc}^*
            G|_{\hat t\times\hat s}\\
  &= G|_{\hat t\times\hat s} - U_{tc} U_{tc}^* G|_{\hat t\times\hat s}
     + U_{tc} (I - \widehat{Q}_{tc} \widehat{Q}_{tc}^*) U_{tc}^*
       G|_{\hat t\times\hat s}\\
  &= \begin{pmatrix}
       G|_{\hat t_1\times\hat s}
       - Q_{t_1 c_1} Q_{t_1 c_1}^* G|_{\hat t_1\times\hat s}\\
       \vdots\\
       G|_{\hat t_\ell\times\hat s}
       - Q_{t_\ell c_\ell} Q_{t_\ell c_\ell}^* G|_{\hat t_\ell\times\hat s}
     \end{pmatrix}
   + U_{tc} (I - \widehat{Q}_{tc} \widehat{Q}_{tc}^*)
     U_{tc}^* G|_{\hat t\times\hat s}.
\end{align*}
The ranges of both terms are perpendicular:
for any pair $x,y\in\bbbc^{\hat s}$ of vectors we have
\begin{align*}
  \langle (G|_{\hat t\times\hat s}
             - U_{tc} U_{tc}^* G|_{\hat t\times\hat s}) x,&
          U_{tc} (I - \widehat{Q}_{tc} \widehat{Q}_{tc}^*)
          U_{tc}^* G|_{\hat t\times\hat s} y \rangle\\
  &= \langle U_{tc}^* (G|_{\hat t\times\hat s}
             - U_{tc} U_{tc}^* G|_{\hat t\times\hat s}) x,
          (I - \widehat{Q}_{tc} \widehat{Q}_{tc}^*)
          U_{tc}^* G|_{\hat t\times\hat s} y \rangle\\
  &= \langle (U_{tc}^* G|_{\hat t\times\hat s}
              - U_{tc}^* G|_{\hat t\times\hat s}) x,
          (I - \widehat{Q}_{tc} \widehat{Q}_{tc}^*)
          U_{tc}^* G|_{\hat t\times\hat s} y \rangle = 0
\end{align*}
due to $U_{tc}^* U_{tc} = I$.
By Pythagoras' theorem, this implies
\begin{align}\label{eq:error_step}
  \|(G|_{\hat t\times\hat s} - Q_{tc} Q_{tc}^* G|_{\hat t\times\hat s})
    x\|_2^2
  &= \sum_{i=1}^\ell \|(G|_{\hat t_i\times\hat s}
             - Q_{t_i c_i} Q_{t_i c_i}^* G|_{\hat t_i\times\hat s})x\|_2^2\\
  &\qquad + \|(I - \widehat{Q}_{tc} \widehat{Q}_{tc}^*)
         U_{tc}^* G|_{\hat t\times\hat s}x\|_2^2
      \qquad\text{ for all } x\in\bbbc^{\hat s},\notag
\end{align}
i.e., we can split the error \emph{exactly} into contributions of the children
and a contribution of the parent $t$.
If the children have children again, we can proceed by induction.
To make this precise, we introduce the sets of \emph{descendants}
\begin{equation*}
  \desc(t,c) := \begin{cases}
    \{ (t,c) \} &\text{ if } \chil(t)=\emptyset,\\
    \{ (t,c) \} \cup \bigcup_{t'\in\chil(t)} \desc(t',\dirchil(t',c))
    &\text{ otherwise}
  \end{cases}
\end{equation*}
for all $t\in\ctI$ and $c\in\mathcal{D}_t$.

%
%
\begin{theorem}[Error representation]
\label{th:error_representation}
We define
\begin{align*}
  \widehat{G}_{tsc} &:= \begin{cases}
    G|_{\hat t\times\hat s} &\text{ if } \chil(t)=\emptyset,\\
    U_{tc}^* G|_{\hat t\times\hat s} &\text{ otherwise}
  \end{cases} &
  &\text{ for all } t,s\in\ctI,\ c\in\mathcal{D}_t.
\end{align*}
In the previous section, we have already defined $\widehat{Q}_{tc}$
for non-leaf clusters.
We extend this notation by setting $\widehat{Q}_{tc} := Q_{tc}$ for
all leaf clusters $t\in\ctI$ and all $c\in\mathcal{D}_t$.
Then we have
\begin{equation*}
  \|(G|_{\hat t\times\hat s} - Q_{tc} Q_{tc}^* G|_{\hat t\times\hat s})
    x\|_2^2
  = \kern-10pt \sum_{(t',c')\in\desc(t,c)}
    \|(\widehat{G}_{t'sc'} - \widehat{Q}_{t'c'} \widehat{Q}_{t'c'}^*
       \widehat{G}_{t'sc'}) x\|_2^2
\end{equation*}
for all $(t,s)\in\lfaII$ with $c=\dirblock(t,s)$ and all $x\in\bbbc^{\hat s}$.
\end{theorem}
\begin{proof}
With the new notation, (\ref{eq:error_step}) takes the form
\begin{align*}
  \|(G|_{\hat t\times\hat s} - Q_{tc} Q_{tc}^* G|_{\hat t\times\hat s})
    x\|_2^2
  &= \sum_{\substack{t'\in\chil(t)\\ c'=\dirchil(t',c)}}^\ell
    \|(G|_{\hat t'\times\hat s} - Q_{t' c'} Q_{t' c'}^*
       G|_{\hat t'\times\hat s})x\|_2^2\\
  &\qquad + \|(\widehat{G}_{tsc} - \widehat{Q}_{tc} \widehat{Q}_{tc}^*
    \widehat{G}_{tsc})x\|_2^2,
\end{align*}
and a straightforward induction yields the result.
\end{proof}

We can see that the matrices $\widehat{G}_{tsc}$ required by this
theorem appear explicitly in the compression algorithm:
$G_{tc}$ is the combination of all matrices $\widehat{G}_{tsc}$
for $s\in\mathcal{R}_{tc}^*$ if $t$ is a leaf, and otherwise
$\widehat{G}_{tc} = U_{tc}^* G_{tc} = \widehat{V}_{tc} Z_{tc}^* P_{tc}^*$
is the combination of all matrices $\widehat{G}_{tsc}$ for
$s\in\mathcal{R}_{tc}^*$.

The compression algorithm computes the singular value decompositions
of the matrices $G_{tc}$ and $\widehat{G}_{tc}$, respectively, so we
have all the singular values at our disposal to guarantee
$\|G_{tc} - Q_{tc} Q_{tc}^* G_{tc}\|_2\leq\epsilon$ or
$\|\widehat{G}_{tc} - \widehat{Q}_{tc} \widehat{Q}_{tc}^*
\widehat{G}_{tc}\|_2\leq\epsilon$, respectively, for any given
accuracy $\epsilon\in\bbbr_{>0}$ by ensuring that the first dropped
singular value $\sigma_{k_{tc}+1}$ is less than $\epsilon$.

\subsection{Block-relative error control}
Although multiple submatrices are combined in $G_{tc}$, they do not
all have to be treated identically \cite[Chapter~6.8]{BO10}:
we can scale the different submatrices with individually chosen
weights, e.g., given $t\in\ctI$ and $c\in\mathcal{D}_t$, we can
choose a weight $\omega_{ts}\in\bbbr_{>0}$ for every
$s\in\mathcal{R}_{tc}^*$ and replace $G_{tc}$ by
\begin{equation*}
  G_{\omega,tc} := \begin{pmatrix}
    \omega_{ts_1}^{-1} G|_{\hat t\times\hat s_1} &
    \ldots &
    \omega_{ts_m}^{-1} G|_{\hat t\times\hat s_m}
  \end{pmatrix}
\end{equation*}
with the enumeration $\mathcal{R}_{tc}^* = \{ s_1,\ldots, s_m \}$.
Correspondingly, $\widehat{G}_{tc}$ is replaced by a weighted version
$\widehat{G}_{\omega,tc}$ and $Z_{tc}$ by $Z_{\omega,tc}$.
The modified algorithm will now guarantee
\begin{align*}
  \|\widehat{G}_{tsc} - \widehat{Q}_{tc} \widehat{Q}_{tc}^* \widehat{G}_{tsc}\|_2
  &= \omega_{ts} \|G_{\omega,tc}|_{\hat t\times\hat s}
                   - Q_{tc} Q_{tc}^* G_{\omega,tc}|_{\hat t\times\hat s}\|_2\\
  &\leq \omega_{ts} \|G_{\omega,tc} - Q_{tc} Q_{tc}^* G_{\omega,tc}\|_2
   \leq \omega_{ts} \epsilon
\intertext{for leaf clusters $t\in\ctI$ and}
  \|\widehat{G}_{tsc} - \widehat{Q}_{tc} \widehat{Q}_{tc}^* \widehat{G}_{tsc}\|_2
  &= \omega_{ts} \|U_t^* G_{\omega,tc}|_{\hat t\times\hat s}
                   - \widehat{Q}_{tc} \widehat{Q}_{tc}^*
                     U_t^* G_{\omega,tc}|_{\hat t\times\hat s}\|_2\\
  &\leq \omega_{ts} \|\widehat{G}_{\omega,tc} - \widehat{Q}_{tc}
    \widehat{Q}_{tc}^* \widehat{G}_{\omega,tc}\|_2
   \leq \omega_{ts} \epsilon
\end{align*}
for non-leaf clusters.
With these modifications, Theorem~\ref{th:error_representation} yields
\begin{align}\label{eq:block_error}
  \|(G|_{\hat t\times\hat s} - Q_{tc} Q_{tc}^* G|_{\hat t\times\hat s})x\|_2^2
  &\leq \sum_{(t',c')\in\desc(t,c)} \omega_{t's}^2 \epsilon^2 &
  &\text{ for all } (t,s)\in\lfaII.
\end{align}
The weights $\omega_{ts}$ can be used to keep the error closely under
control.
As an example, we consider how to implement \emph{block-relative}
error controls, i.e., how to ensure
\begin{equation*}
  \|G|_{\hat t\times\hat s} - Q_{tc} Q_{tc}^* G|_{\hat t\times\hat s}\|_2
  \leq \epsilon \|G|_{\hat t\times\hat s}\|
\end{equation*}
for a block $(t,s)\in\lfaII$.
We start by observing that we have
\begin{equation*}
  \|G|_{\hat t\times\hat s}\|_2
  = \|V_{tc} S_{ts} V_{sc}^*\|_2
  = \|R_{tc} S_{ts} R_{sc}^*\|_2
\end{equation*}
due to the admissibility and using the basis weights introduced in
Definition~\ref{de:basis_weights}, so the spectral norm can be
computed efficiently.

We assume that a cluster can have at most $m$ children and set
\begin{align*}
  \omega_{t's}
  &:= \begin{cases}
        \frac{1}{\sqrt{m+1}} \|G|_{\hat t\times\hat s}\|_2
        &\text{ if } t'=t,\\
        \frac{1}{\sqrt{m+1}} \omega_{t^+s}
        &\text{ if } t'\in\chil(t^+)
      \end{cases} &
  &\text{ for all } (t',c')\in\desc(t,c).
\end{align*}
Keeping in mind that every cluster can have at most $m$ children,
substituting $\omega_{t's}$ in (\ref{eq:block_error}) and summing up
level by level yields
\begin{equation*}
  \|G|_{\hat t\times\hat s} - Q_{tc} Q_{tc}^* G|_{\hat t\times\hat s}\|_2^2
  \leq \sum_{(t',c')\in\desc(t,c)} \omega_{t's}^2 \epsilon^2
  \leq \sum_{\ell=0}^\infty \left(\frac{m}{m+1}\right)^\ell
                 \omega_{ts}^2 \epsilon^2,
\end{equation*}
allowing us to evaluate the geometric sum to conclude
\begin{equation*}
  \|G|_{\hat t\times\hat s} - Q_{tc} Q_{tc}^* G|_{\hat t\times\hat s}\|_2^2
  \leq \frac{1}{1 - \frac{m}{m+1}} \omega_{ts}^2 \epsilon^2
  = (m+1) \omega_{ts}^2 \epsilon^2 = \epsilon^2 \|G|_{\hat t\times\hat s}\|_2^2.
\end{equation*}
The weights $\omega_{t's}$ can be computed and conveniently included
during the construction of the total weights at only minimal additional
cost.

\subsection{Stability}
In order to improve the efficiency of our algorithm, we would like
to replace the $\mathcal{DH}^2$-matrix $G$ by an approximation.
If we want to ensure that the result of the compression algorithm
is still useful, we have to investigate its stability.
In the following, $G$ denotes the matrix treated during the compression
algorithm, while $H$ denotes the matrix that we actually want to approximate.

%
%
\begin{lemma}[Stability]
Let $H\in\bbbc^{\Idx\times\Idx}$.
We have
\begin{align*}
  \|(H|_{\hat t\times\hat s} - Q_{tc} Q_{tc}^* H|_{\hat t\times\hat s})x\|_2
  &\leq \|(G|_{\hat t\times\hat s}
               - Q_{tc} Q_{tc}^* G|_{\hat t\times\hat s}x\|_2
   + \|(H|_{\hat t\times\hat s} - G|_{\hat t\times\hat s})x\|_2
\end{align*}
for all $(t,s)\in\lfaII$ with $c=\dirblock(t,s)$ and all $x\in\bbbc^{\hat s}$.
\end{lemma}
\begin{proof}
Let $(t,s)\in\lfaII$, $c=\dirblock(t,s)$ and $x\in\bbbc^{\hat s}$.
We have
\begin{align*}
  H|_{\hat t\times\hat s} - Q_{tc} Q_{tc}^* H|_{\hat t\times\hat s}
  &= G|_{\hat t\times\hat s} + (H-G)|_{\hat t\times\hat s}
     - Q_{tc} Q_{tc}^* G|_{\hat t\times\hat s}
     - Q_{tc} Q_{tc}^* (H-G)|_{\hat t\times\hat s}\\
  &= G|_{\hat t\times\hat s} - Q_{tc} Q_{tc}^* G|_{\hat t\times\hat s}
   + (I - Q_{tc} Q_{tc}^*) (H-G)|_{\hat t\times\hat s}
\end{align*}
and therefore by the triangle inequality
\begin{equation*}
  \|(H|_{\hat t\times\hat s} - Q_{tc} Q_{tc}^* H|_{\hat t\times\hat s})x\|_2
  \leq \|(G|_{\hat t\times\hat s} - 
             Q_{tc} Q_{tc}^* G|_{\hat t\times\hat s})x\|_2
     + \|(I - Q_{tc} Q_{tc}^*) (H-G)x\|_2.
\end{equation*}  
We let $y := (H-G)|_{\hat t\times\hat s} x$ and make use of the isometry of
$Q_{tc}$ to find
\begin{align*}
  \|(I-Q_{tc} Q_{tc}^*) y\|_2^2
  &= \|y\|_2^2 - \langle y, Q_{tc} Q_{tc}^* y \rangle_2
     - \langle Q_{tc} Q_{tc}^* y, y \rangle_2
     + \|Q_{tc} Q_{tc}^* y\|_2^2\\
  &= \|y\|_2^2 - 2 \langle y, Q_{tc} Q_{tc}^* Q_{tc} Q_{tc}^* y \rangle_2
     + \|Q_{tc} Q_{tc}^* y\|_2^2\\
  &= \|y\|_2^2 - 2 \langle Q_{tc} Q_{tc}^* y, Q_{tc} Q_{tc}^* y \rangle_2
     + \|Q_{tc} Q_{tc}^* y\|_2^2\\
  &= \|y\|_2^2 - \|Q_{tc} Q_{tc}^* y\|_2^2 \leq \|y\|_2^2.
\end{align*}
Since this is equivalent with $\|(I-Q_{tc} Q_{tc}^*)(H-G)|_{\hat t\times\hat s} x\|_2^2
\leq \|(H-G)|_{\hat t\times\hat s} x\|_2^2$, the proof is complete.
\end{proof}

This lemma implies that if we want to approximate a matrix $H$, but
apply the algorithm to an approximation $G$ satisfying the block-relative
error estimate
\begin{align*}
  \|H|_{\hat t\times\hat s} - G|_{\hat t\times\hat s}\|_2
  &\leq \epsilon \|H|_{\hat t\times\hat s}\|_2 &
  &\text{ for all } (t,s)\in\lfaII,
\end{align*}
we will obtain
\begin{align*}
  \|H|_{\hat t\times\hat s} - Q_{tc} Q_{tc}^* H|_{\hat t\times\hat s}\|_2
  &\leq \|G|_{\hat t\times\hat s}
           - Q_{tc} Q_{tc}^* G|_{\hat t\times\hat s}\|_2
        + \|H|_{\hat t\times\hat s} - G|_{\hat t\times\hat s}\|_2\\
  &\leq \epsilon \|G|_{\hat t\times\hat s}\|_2
        + \epsilon \|H|_{\hat t\times\hat s}\|_2\\
  &\leq \epsilon \bigl(\|H|_{\hat t\times\hat s}\|_2
                           + \|G-H|_{\hat t\times\hat s}\|_2\bigr)
        + \epsilon \|H|_{\hat t\times\hat s}\|_2\\
  &\leq \epsilon (1+\epsilon) \|H|_{\hat t\times\hat s}\|_2
        + \epsilon \|H|_{\hat t\times\hat s}\|_2\\
  &= \epsilon(2+\epsilon) \|H|_{\hat t\times\hat s}\|_2,
\end{align*}
i.e., the basis constructed to ensure block-relative error estimates
for the matrix $G$ will also ensure block-relative error estimates
for the matrix $H$, only with a slightly larger error factor
$\approx 2\epsilon$.
Since our error-control strategy can ensure any accuracy $\epsilon>0$,
this is quite satisfactory.

\section{Approximated weights}
\label{se:weights}

Figure~\ref{fi:recomp} suggests that for higher accuracies,
the basis weights $(R_{sc})_{s\in\ctI,\ c\in\mathcal{D}_s}$ can require
more storage than the entire recompressed $\mathcal{DH}^2$-matrix.
With the error representation of Theorem~\ref{th:error_representation}
and the stability analysis of the previous section at our disposal, we
can investigate ways to reduce the storage requirements without causing
significant harm to the final result.

We do not have to worry about the total weights
$(Z_{tc})_{t\in\ctI, c\in\mathcal{D}_t}$, since they can be set up
during the recursive construction of the adaptive cluster basis.

\subsection{Direct approximation of weights.}
The basis weight matrices $R_{sc}$ are required by our algorithm
when it sets up the total weight matrix $Z_{tc}$ with
\begin{equation*}
  G_{tc} = V_{tc} Z_{tc}^* P_{tc}^*
\end{equation*}
using
\begin{equation*}
  G_{tc}|_{\hat t\times\hat s}
  = V_{tc} S_{ts} W_{sc}^*
  = V_{tc} S_{ts} R_{sc}^* Q_{sc}^*
\end{equation*}
for an admissible block $b=(t,s)\in\lfaII$.
The isometric matrix $Q_{sc}$ influences only $P_{tc}$ and can be
dropped since it does not influence the singular values or the
left singular vectors.

Our goal is to replace the basis weight $R_{sc}$ by an approximation
$\widetilde{R}_{sc}$ while ensuring that the recompression algorithm
keeps working reliably.
We find
\begin{equation*}
  \|G_{tc}|_{\hat t\times\hat s}
    - V_{tc} S_{ts} \widetilde{R}_{sc}^* Q_{sc}^*\|_2
  = \|V_{tc} S_{ts} (R_{sc}^* - \widetilde{R}_{sc}^*) Q_{sc}^*\|_2
  = \|V_{tc} S_{ts} (R_{sc} - \widetilde{R}_{sc})^*\|_2
\end{equation*}
and conclude that it is sufficient to ensure that the product
$\widetilde{R}_{sc} S_{ts}^*$ is a good approximation of the product
$R_{sc} S_{ts}^*$, we do not require $\widetilde{R}_{sc}$ itself to be
a good approximation of $R_{sc}$.
This is a crucial observation, because important approximation properties
are due to the kernel function represented by $S_{ts}$, not due to the
essentially arbitrary polynomial basis represented by $R_{sc}$.

Since the basis weight $R_{sc}$ will be used for multiple clusters
$t\in\ctI$, we introduce the sets
\begin{align}\label{eq:Csc_definition}
  \mathcal{C}_{sc} &:= \{ t\in\ctI\ :\ (t,s)\in\lfaII,
           \ \dirblock(t,s)=c \} &
  &\text{ for all } s\in\ctI,\ c\in\mathcal{D}_s
\end{align}
in analogy to the sets $\mathcal{R}_{tc}$ used for the compression
algorithm in (\ref{eq:Rtc_definition}).
Enumerating the elements by $\mathcal{C}_{sc} = \{t_1,\ldots,t_m\}$
leads us to consider the approximation of the matrix
\begin{equation}\label{eq:Wsc_def}
  W_{sc} := R_{sc} \begin{pmatrix}
                    S_{t_1s}^* & \ldots & S_{t_ms}^*
                  \end{pmatrix}.
\end{equation}
The optimal solution is again provided by the singular value
decomposition of $W_{sc}$:
for the singular values $\sigma_1,\sigma_2,\ldots$ and a given
accuracy $\epsilon\in\bbbr_{>0}$, we choose a rank $k_{sc}\in\bbbn$
such that $\sigma_{k_{sc}+1} \leq \epsilon$ and combine the first
$k_{sc}$ left singular vectors in an isometric matrix $\widetilde{Q}_{sc}$.
We use the corresponding low-rank approximation
$\widetilde{R}_{sc} := \widetilde{Q}_{sc} \widetilde{Q}_{sc}^* R_{sc}$ and find
\begin{equation*}
  \|R_{sc} S_{st}^* - \widetilde{R}_{sc} S_{st}^*\|_2
  = \|R_{sc} S_{st}^* - \widetilde{Q}_{sc} \widetilde{Q}_{sc}^* R_{sc} S_{st}^*\|_2
  \leq \|W_{sc} - \widetilde{Q}_{sc} \widetilde{Q}_{sc}^* W_{sc}\|_2
  \leq \epsilon.
\end{equation*}
The resulting algorithm is summarized in Figure~\ref{fi:approx_weights}.
It is important to note that the original basis weights $R_{sc}$ are
discarded as soon as they are no longer needed so that the original weights
have to be kept in storage only for the children of the current cluster
and the children of its ancestors at every point of the algorithm.

For the basis construction algorithm, cf. Figure~\ref{fi:build_basis},
only the coefficient matrices $\widehat{R}_{sc} := \widetilde{Q}_{sc}^*
R_{sc}$ are required, since the isometric matrices $\widetilde{Q}_{sc}$
do not influence the singular values and the left singular vectors.
Our algorithm only provides these matrices to save storage.

%
%
\begin{figure}
  \begin{quotation}
    \begin{tabbing}
      \textbf{procedure} approx\_weights($s$);\\
      \textbf{begin}\\
      \quad\= \textbf{if} $\chil(s)=\emptyset$ \textbf{then}\\
      \> \quad\= \textbf{for} $c\in\mathcal{D}_s$ \textbf{do begin}\\
      \> \> \quad\= Find a thin Householder decomposition
                        $V_{sc} = Q_{sc} R_{sc}$;\\
      \> \> \> Set up $W_{sc}$ as in (\ref{eq:Wsc_def}) or (\ref{eq:Wsc_omega});\\
      \> \> \> Compute the singular value decomposition
                   $W_{sc} = U \Sigma V^*$;\\
      \> \> \> Choose a rank $k_{sc}$, shrink $U$ to its first $k_{sc}$ columns;\\
      \> \> \> $\widehat{R}_{sc} \gets U^* R_{sc}$\\
      \> \> \textbf{end}\\
      \> \textbf{else begin}\\
      \> \> \textbf{for} $s'\in\chil(s)$ \textbf{do}\\
      \> \> \> approx\_weights($s'$);\\
      \> \> \textbf{for} $c\in\mathcal{D}_s$ \textbf{do begin}\\
      \> \> \> Set up $\widehat{V}_{sc}\in\bbbc^{M_{sc}\times k}$ as
                   in (\ref{eq:Vhat_sc});\\
      \> \> \> Find a thin Householder decomposition
                        $\widehat{V}_{sc} = \widehat{Q}_{sc} R_{sc}$;\\
      \> \> \> Set up $W_{sc}$ as in (\ref{eq:Wsc_def}) or (\ref{eq:Wsc_omega});\\
      \> \> \> Compute the singular value decomposition
                   $W_{sc} = U \Sigma V^*$;\\
      \> \> \> Choose a rank $k_{sc}$, shrink $U$ to its first $k_{sc}$ columns;\\
      \> \> \> $\widehat{R}_{sc} \gets U^* R_{sc}$\\
      \> \> \textbf{end};\\
      \> \> \textbf{for} $s'\in\chil(s)$, $c'\in\mathcal{D}_{s'}$ \textbf{do}\\
      \> \> \> Discard $R_{s'c'}$ from memory\\
      \> \textbf{end}\\
      \textbf{end}
    \end{tabbing}
  \end{quotation}
  \caption{Construction of the approximated basis weights $\widetilde{R}_{sc}=
    \widetilde{Q}_{sc} \widehat{R}_{sc}$}
  \label{fi:approx_weights}
\end{figure}

\subsection{Block-relative error control}

Again, we are interested in blockwise relative error estimates,
and as before, we can modify the blockwise approximation by
introducing weights $\omega_{ts}\in\bbbr_{>0}$ and considering
\begin{equation}\label{eq:Wsc_omega}
  W_{\omega,sc} := R_{sc} \begin{pmatrix}
                          \omega_{t_1s}^{-1} S_{t_1s}^* & \ldots
                          & \omega_{t_ms}^{-1} S_{t_m s}^*
                        \end{pmatrix}.
\end{equation}
Replacing $W_{sc}$ by $W_{\omega,sc}$ yields
\begin{align*}
  \|R_{sc} S_{st}^* - \widetilde{R}_{sc} S_{st}^*\|_2
  &= \omega_{ts} \|R_{sc} \omega_{ts}^{-1} S_{st}^*
                    - \widetilde{Q}_{sc} \widetilde{Q}_{sc}^* R_{sc} \omega_{ts}^{-1} S_{st}^*\|_2\\
  &\leq \omega_{ts} \|W_{\omega,sc} - \widetilde{Q}_{sc} \widetilde{Q}_{sc}^* W_{\omega,sc}\|_2
   \leq \omega_{ts} \epsilon.
\end{align*}
For the blockwise error we obtain
\begin{align*}
  \|G|_{\hat t\times\hat s} - V_{tc} S_{ts} \widetilde{R}_{sc}^* Q_{sc}^*\|_2
  &= \|V_{tc} S_{ts} W_{sc}^* - V_{tc} S_{tc} \widetilde{R}_{sc}^* Q_{sc}^*\|_2\\
  &= \|V_{tc} S_{ts} (R_{sc} - \widetilde{R}_{sc})^* Q_{sc}^*\|_2\\
  &\leq \|V_{tc}\|_2 \|Q_{sc} (R_{sc} - \widetilde{R}_{sc}) S_{ts}^*\|_2\\
  &= \|V_{tc}\|_2 \|(R_{sc} - \widetilde{R}_{sc}) S_{ts}^*\|_2
   \leq \|V_{tc}\|_2 \omega_{ts} \epsilon,
\end{align*}
so a relative error bound is guaranteed if we ensure
\begin{equation*}
  \omega_{ts} \leq \frac{\|V_{tc} S_{ts} V_{sc}^*\|_2}{\|V_{tc}\|_2}.
\end{equation*}
Evaluating the numerator and denominator exactly would require
us again to have the full basis weights at our disposal.
Fortunately, a projection is sufficient for our purposes:
in a preparation step, we compute the basis weight $R_{tc}$,
find its singular value decomposition and use a given number
$k_\text{norm}\in\bbbn$ of left singular vectors to form an auxiliary
isometric matrix $P_{tc}$ and store $N_{tc} := P_{tc}^* R_{rc}$.
Since the first singular value corresponds to the spectral
norm of $R_{tc}$, and therefore the spectral norm of $V_{tc}$,
we have
\begin{equation*}
  \|N_{tc}\|_2 = \|P_{tc}^* R_{tc}\|_2 = \|R_{tc}\|_2 = \|V_{tc}\|_2
\end{equation*}
and can evaluate the denominator exactly.
Since $P_{tc} P_{tc}^*$ is an orthogonal projection, we also have
\begin{equation*}
  \|V_{tc} S_{ts} V_{sc}^*\|_2
  = \|R_{tc} S_{ts} R_{sc}^*\|_2
  \geq \|P_{tc}^* R_{tc} S_{ts} R_{sc}^*\|_2
  = \|N_{tc} S_{ts} R_{sc}^*\|_2,
\end{equation*}
i.e., we can find a lower bound for the numerator.
Fortunately, a lower bound is sufficient for our purposes, and
we can use
\begin{equation*}
  \omega_{ts} := \frac{\|N_{tc} S_{ts} R_{sc}^*\|_2}{\|N_{tc}\|_2}
              \leq \frac{\|V_{tc} S_{ts} V_{sc}^*\|_2}{\|V_{tc}\|_2}.
\end{equation*}
The algorithm for constructing the norm-estimation matrices
$N_{tc}$ is summarized in Figure~\ref{fi:approx_norms}.
It uses exact Householder factorizations $Q_{tc} R_{tc} = V_{tc}$
for all basis matrices and then truncates $R_{tc}$.
In order to make the computation efficient, we use
\begin{align}\label{eq:Vhat_tc_exact}
  V_{tc}
  &= \begin{pmatrix}
       V_{t_1,c_1} E_{t_1 c}\\
       \vdots\\
       V_{t_n,c_n} E_{t_n c}
     \end{pmatrix}
   = \begin{pmatrix}
       Q_{t_1,c_1} & & \\
       & \ddots & \\
       & & Q_{t_n c_n}
     \end{pmatrix} \widehat{V}_{tc}, &
  \widehat{V}_{tc} &:= \begin{pmatrix}
    R_{t_1,c_1} E_{t_1 c}\\
    \vdots\\
    R_{t_n,c_n} E_{t_n c}
  \end{pmatrix}
\end{align}
to replace $V_{tc}$ by the projected matrix $\widehat{V}_{tc}$
with $\chil(t)=\{t_1,\ldots,t_n\}$ and $c_i=\dirchil(t_i,c)$
for all $i\in[1:n]$.

Figure~\ref{fi:compweight} shows that compressing the basis
weights following these principles leaves the accuracy of the
matrix intact and significantly reduces the storage requirements.

%
%
\begin{figure}
  \begin{quotation}
    \begin{tabbing}
      \textbf{procedure} approx\_norms($t$);\\
      \textbf{begin}\\
      \quad\= \textbf{if} $\chil(t)=\emptyset$ \textbf{then}\\
      \> \quad\= \textbf{for} $c\in\mathcal{D}_t$ \textbf{do begin}\\
      \> \> \quad\= Find a thin Householder decomposition
                        $V_{tc} = Q_{tc} R_{tc}$;\\
      \> \> \> Compute the singular value decomposition
                        $R_{tc} = U \Sigma V^*$;\\
      \> \> \> Shrink $U$ to its first $k_\text{norm}$ columns;\\
      \> \> \> $N_{tc} \gets U^* R_{tc}$\\
      \> \> \textbf{end}\\
      \> \textbf{else begin}\\
      \> \> \textbf{for} $t'\in\chil(t)$ \textbf{do}\\
      \> \> \> approx\_norms($t'$);\\
      \> \> \textbf{for} $c\in\mathcal{D}_t$ \textbf{do begin}\\
      \> \> \> Set up $\widehat{V}_{tc}$ as in (\ref{eq:Vhat_tc_exact});\\
      \> \> \> Find a thin Householder decomposition
                        $\widehat{V}_{tc} = \widehat{Q}_{tc} R_{tc}$\\
      \> \> \> Compute the singular value decomposition
                        $R_{tc} = U \Sigma V^*$;\\
      \> \> \> Shrink $U$ to its first $k_\text{norm}$ columns;\\
      \> \> \> $N_{tc} \gets U^* R_{tc}$\\
      \> \> \textbf{end};\\
      \> \> \textbf{for} $t'\in\chil(t)$, $c'\in\mathcal{D}_{t'}$ \textbf{do}\\
      \> \> \> Discard $R_{t'c'}$ from memory\\
      \> \textbf{end}\\
      \textbf{end}
    \end{tabbing}
  \end{quotation}
  \caption{Construction of norm-approximation matrices $N_{tc}$}
  \label{fi:approx_norms}
\end{figure}

%
%
\begin{figure}
  \begin{center}
    \includegraphics[width=0.45\textwidth]{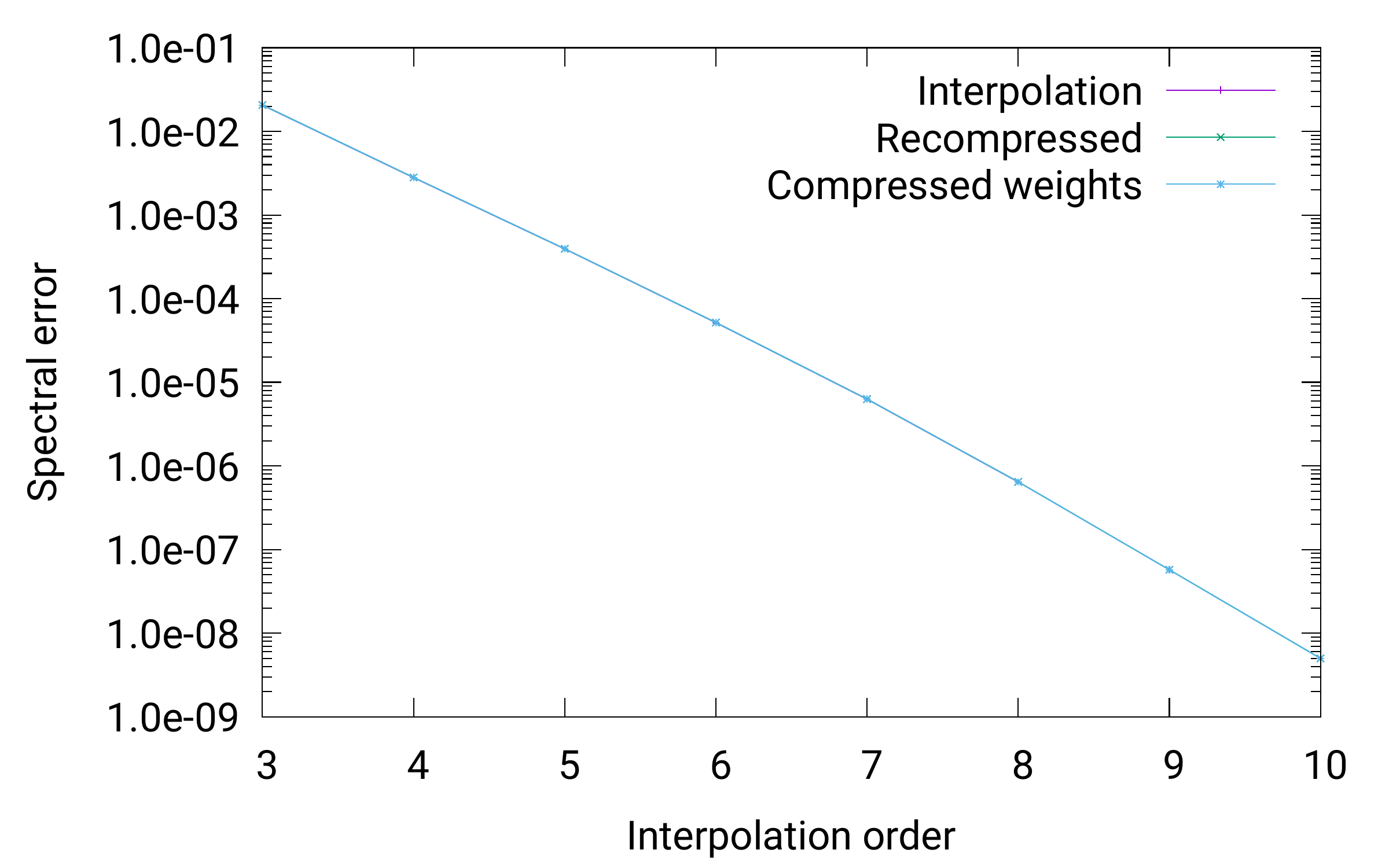}%
    \qquad%
    \includegraphics[width=0.45\textwidth]{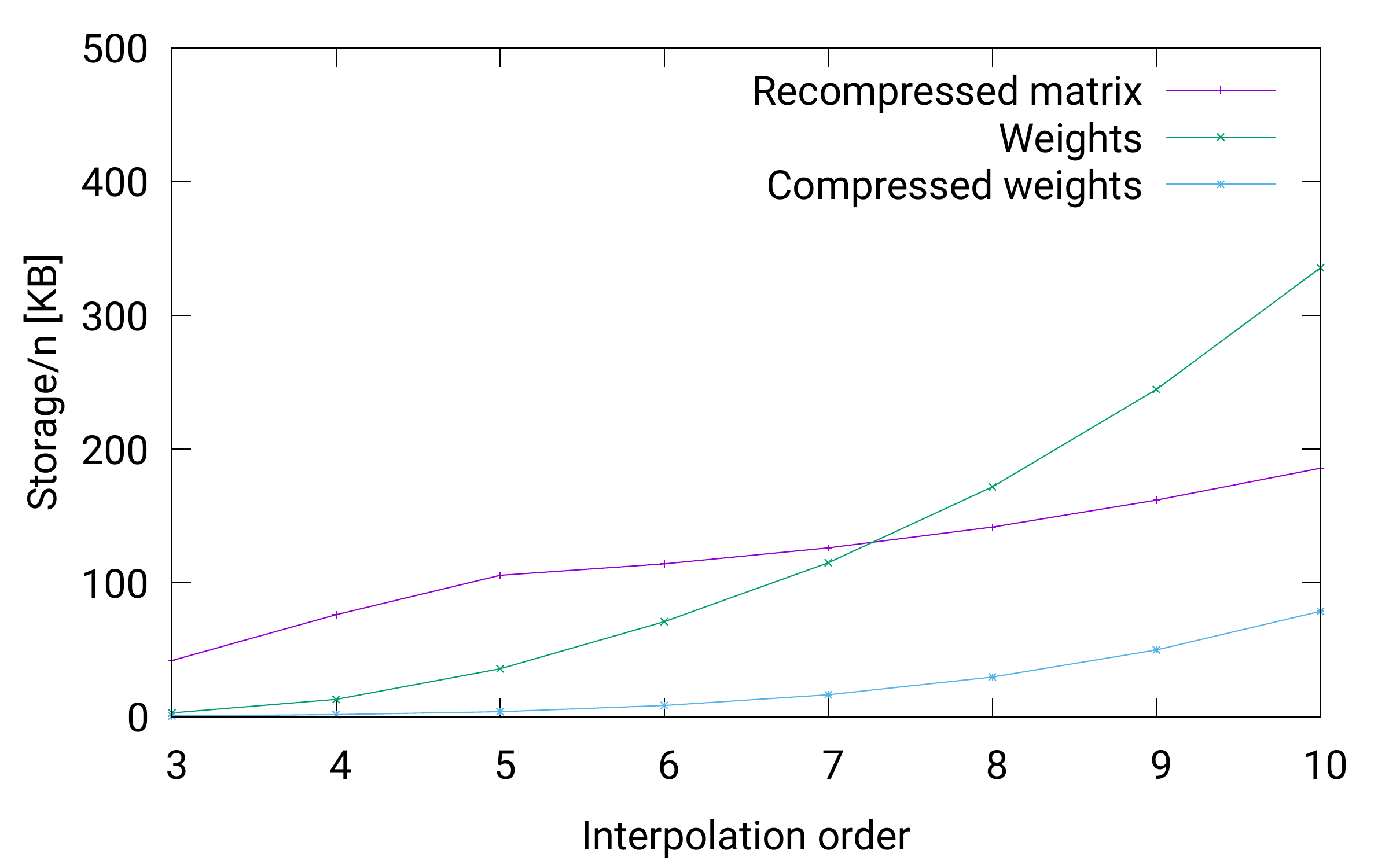}
  \end{center}
 
  \caption{Left: Convergence of recompressed interpolation with
    compressed weights.
    Right: Storage requirements of compressed and uncompressed weights}
  \label{fi:compweight}
\end{figure}

\section{Numerical experiments}
\label{se:experiments}

To demonstrate the properties of the new algorithms in practical
applications, we consider direct boundary integral formulations
for the Helmholtz problem on the unit sphere.
We create a mesh for the unit sphere by starting with a double
pyramid $P=\{x\in\bbbr^3\ :\ |x_1|+|x_2|+|x_3|=1\}$ and refining
each of its faces into $m^2$ triangles, where
$m\in\{16,24,32,48,\ldots,1024\}$.
Projecting these triangles' vertices to the unit sphere yields
regular surface meshes with between $2\,048$ and $8\,388\,608$
triangles.

We discretize the direct boundary integral formulations for
the Dirichlet-to-Neumann and the Neumann-to-Dirichlet problem
with piecewise constant basis functions for the Neumann values
and continuous linear nodal basis functions for the Dirichlet
values.
The approximation of the single-layer matrix by a
$\mathcal{DH}^2$-matrix has already been discussed.

For the double-layer matrix, we apply directional interpolation to
the kernel function and take the normal derivative of the result.
This again yields an $\mathcal{DH}^2$-matrix.
The approximation error has been investigated in \cite{BO22}.

For the hypersingular matrix, we use partial integration
\cite[Corollary~3.3.24]{SASC11} and again apply directional
interpolation to the remaining kernel function.

In order to save storage, basis weights for the row and column
basis of the single-layer matrix and the row basis of the double-layer
matrix are shared, and basis weights for the column basis of the
double-layer matrix and the row and column basis of the hypersingular
matrix are also shared.
In our implementation, this can be easily accomplished by including
more matrix blocks in the matrices $W_{sc}$.

The resulting systems of linear equations are solved by a
GMRES method that is preconditioned using an $\mathcal{H}$-LU
factorization \cite[Chapter~7.6]{HA15} of a coarse approximation
of the $\mathcal{DH}^2$-matrix.

%
%
\begin{table}
  \begin{equation*}
    \begin{array}{rr|rrr|rrr|rrr}
      && \multicolumn{3}{c|}{\text{Weight } \varphi_i}
      & \multicolumn{3}{c|}{\text{SLP}}
      & \multicolumn{3}{c}{\text{DLP}}\\
      n & m & \text{Norm} & \text{Comp} & \text{Mem}
      & \text{Row} & \text{Col} & \text{Mem}
      & \text{Row} & \text{Col} & \text{Mem}\\
     \hline
     8\,192 & 3 & <0.1 & 2.6 & 5 & 0.2 & 0.3 & 319 & 0.1 & 0.1 & 331\\
     18\,432 & 4 & 0.2 & 11.5 & 37 & 2.0 & 2.1 & 588 & 1.4 & 1.4 & 648\\
     32\,768 & 4 & 0.2 & 8.7 & 80 & 3.0 & 3.0 & 943 & 1.6 & 2.7 & 951\\
     73\,728 & 5 & 0.9 & 70.5 & 431 & 5.7 & 5.8 & 2\,013 & 4.4 & 4.3 & 1\,792\\
     131\,072 & 5 & 1.3 & 101.3 & 853 & 9.1 & 9.4 & 3\,594 & 7.0 & 6.5 & 3\,243\\
     294\,912 & 5 & 3.6 & 183.1 & 2\,072 & 20.7 & 20.9 & 8\,234 & 14.7 & 13.7 & 7\,774\\
     524\,288 & 6 & 6.5 & 447.7 & 7\,217 & 71.1 & 52.6 & 15\,797 & 44.1 & 37.4 & 13\,959\\
     1\,179\,648 & 6 & 14.1 & 824.2 & 17\,091 & 156.9 & 119.6 & 37\,949 & 100.7 & 71.5 & 35\,281\\
     2\,097\,152 & 7 & 26.1 & 3\,310.3 & 51\,516 & 1\,062.4 & 718.8 & 71\,768 & 602.6 & 513.3 & 62\,032\\
     4\,718\,592 & 7 & 57.8 & 7\,034.6 & 122\,268 & 2\,704.3 & 1\,858.2 & 176\,761 & 1\,606.0 & 1\,316.2 & 162\,386\\
     8\,388\,608 & 7 & 101.9 & 12\,041.2 & 224\,330 & 4\,821.9 & 3\,636.7 & 332\,950 & 2\,799.6 & 2\,616.1 & 323\,817
    \end{array}
  \end{equation*}

  \caption{Helmholtz boundary integral equation with constant wave number $\kappa=4$}
  \label{ta:constant}
\end{table}

Table~\ref{ta:constant} contains results for a first experiment with the
constant wave number $\kappa=4$.  The column ``n'' gives the number of
triangles, the column ``m'' gives the order of the interpolation, the column
``Norm'' gives the time in seconds for the approximation of the matrix norm
with the algorithm given in Figure~\ref{fi:approx_norms}, the column ``Comp''
gives the time for the compression of the weights by the algorithm in
Figure~\ref{fi:approx_weights} and ``Mem'' the storage requirements in MB
for the compressed weight matrices.

The columns ``Row'' and ``Col'' give the times in seconds for constructing the
adaptive row and column cluster bases, both for the single-layer and the
double-layer matrix, while the columns ``Mem'' give the storage requirements
in MB for the compressed $\mathcal{DH}^2$-matrices.

The experiment was performed on a server with two AMD EPYC 7713 processors
with $64$ cores each and a total of $2\,048$ GB of memory.

We can see that the runtimes and storage requirements grow slowly with
increasing matrix dimension and increasing polynomial order, as predicted by
the theory.  Truncation tolerances were chosen to ensure that the convergence
of the original Galerkin discretization is preserved, i.e., we obtain
$L^2$-norm errors falling like $\mathcal{O}(h)$ for the Dirichlet-to-Neumann
problem and like $\mathcal{O}(h^2)$ for the Neumann-to-Dirichlet problem.

%
%
\begin{table}
  \begin{equation*}
    \begin{array}{rrr|rrr|rrr|rrr}
      &&& \multicolumn{3}{c|}{\text{Weight } \varphi_i}
      & \multicolumn{3}{c|}{\text{SLP}}
      & \multicolumn{3}{c}{\text{DLP}}\\
      n & \kappa & m & \text{Norm} & \text{Comp} & \text{Mem}
      & \text{Row} & \text{Col} & \text{Mem}
      & \text{Row} & \text{Col} & \text{Mem}\\
     \hline
     8\,192 & 4 & 3 & <0.1 & 2.3 & 5 & 0.8 & 0.9 & 319 & 0.3 & 0.4 & 331\\
     18\,432 & 6 & 4 & 0.2 & 12.3 & 35 & 3.2 & 3.2 & 993 & 1.7 & 2.0 & 1\,057\\
     32\,768 & 8 & 4 & 0.3 & 33.8 & 77 & 7.2 & 7.4 & 2\,295 & 4.6 & 4.5 & 2\,336\\
     73\,728 & 12 & 5 & 0.9 & 246.4 & 428 & 24.4 & 25.0 & 7\,318 & 19.4 & 18.7 & 7\,507\\
     131\,072 & 16 & 5 & 1.6 & 559.9 & 966 & 50.1 & 47.2 & 16\,888 & 42.8 & 37.8 & 17\,170\\
     294\,912 & 24 & 5 & 5.6 & 1\,579.0 & 3\,599 & 125.5 & 123.6 & 43\,792 & 107.8 & 103.0 & 45\,345\\
     524\,288 & 32 & 6 & 13.2 & 6\,815.1 & 14\,383 & 474.0 & 331.5 & 94\,766 & 489.1 & 389.5 & 93\,371\\
     1\,179\,648 & 48 & 6 & 34.7 & 19\,657.0 & 44\,523 & 1\,324.8 & 1\,008.2 & 255\,373 & 1\,540.6 & 979.9 & 259\,687\\
     2\,097\,152 & 64 & 7 & 105.5 & 84\,707.6 & 174\,972 & 7\,622.8 & 6\,252.1 & 455\,584 & 8\,541.0 & 6\,106.9 & 438\,976
    \end{array}
  \end{equation*}

  \caption{Helmholtz boundary integral equation with growing wave number}
  \label{ta:growing}
\end{table}

Table~\ref{ta:growing} contains results for a second experiment with the a
wave number that grows as the mesh is refined.  This is common in practical
applications when the mesh is chosen just fine enough to resolve waves.

The growing wave number makes it significantly harder to satisfy the
admissibility condition (\ref{eq:admissibility_1}) and thereby leads to an
increase in blocks that have to be treated.  The admissibility condition
(\ref{eq:admissibility_2}) implies that we also have to introduce a growing
number of directions as the wave number increases.

%
%
\begin{figure}
  \begin{center}
    \includegraphics[width=0.45\textwidth]{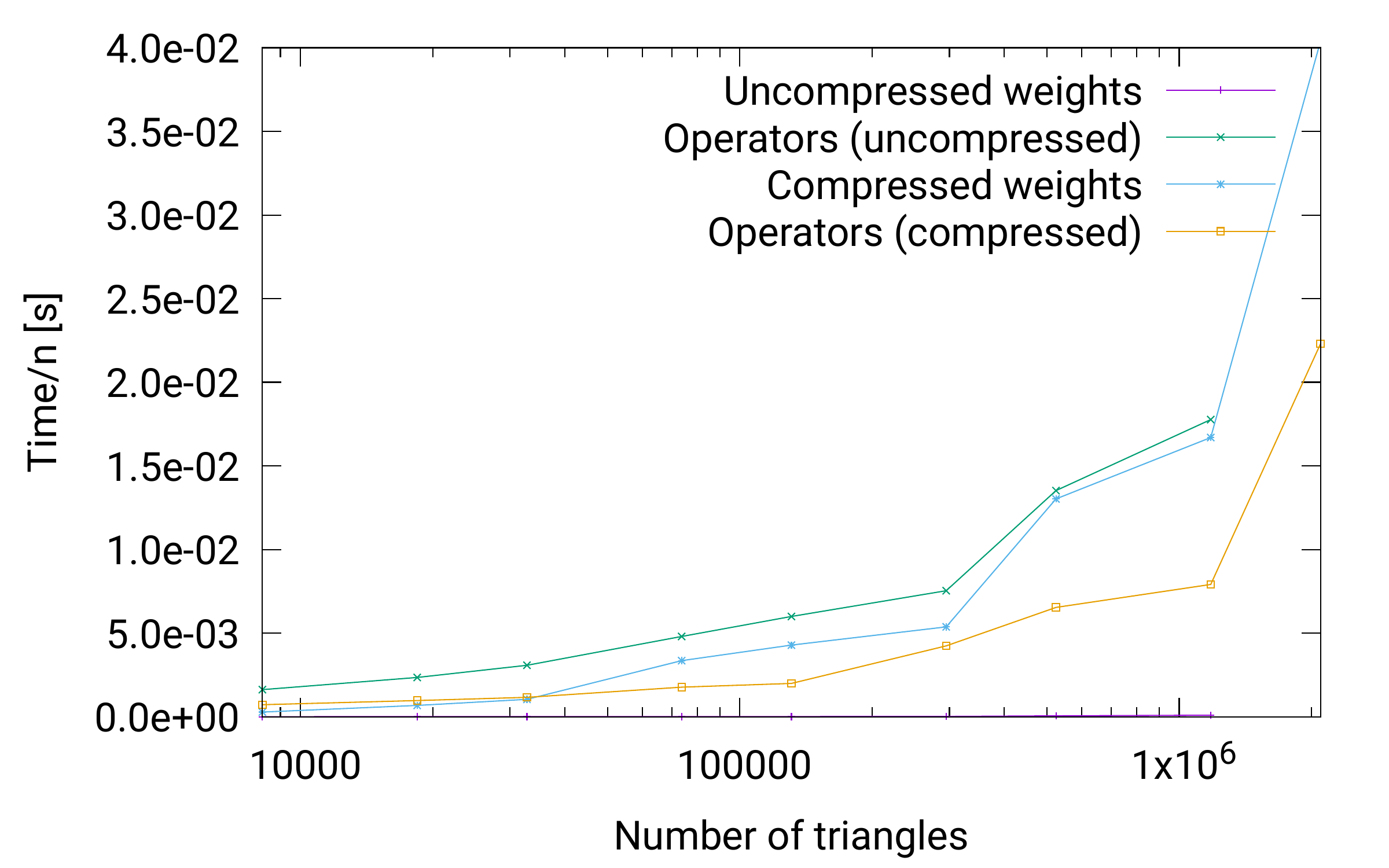}%
    \qquad%
    \includegraphics[width=0.45\textwidth]{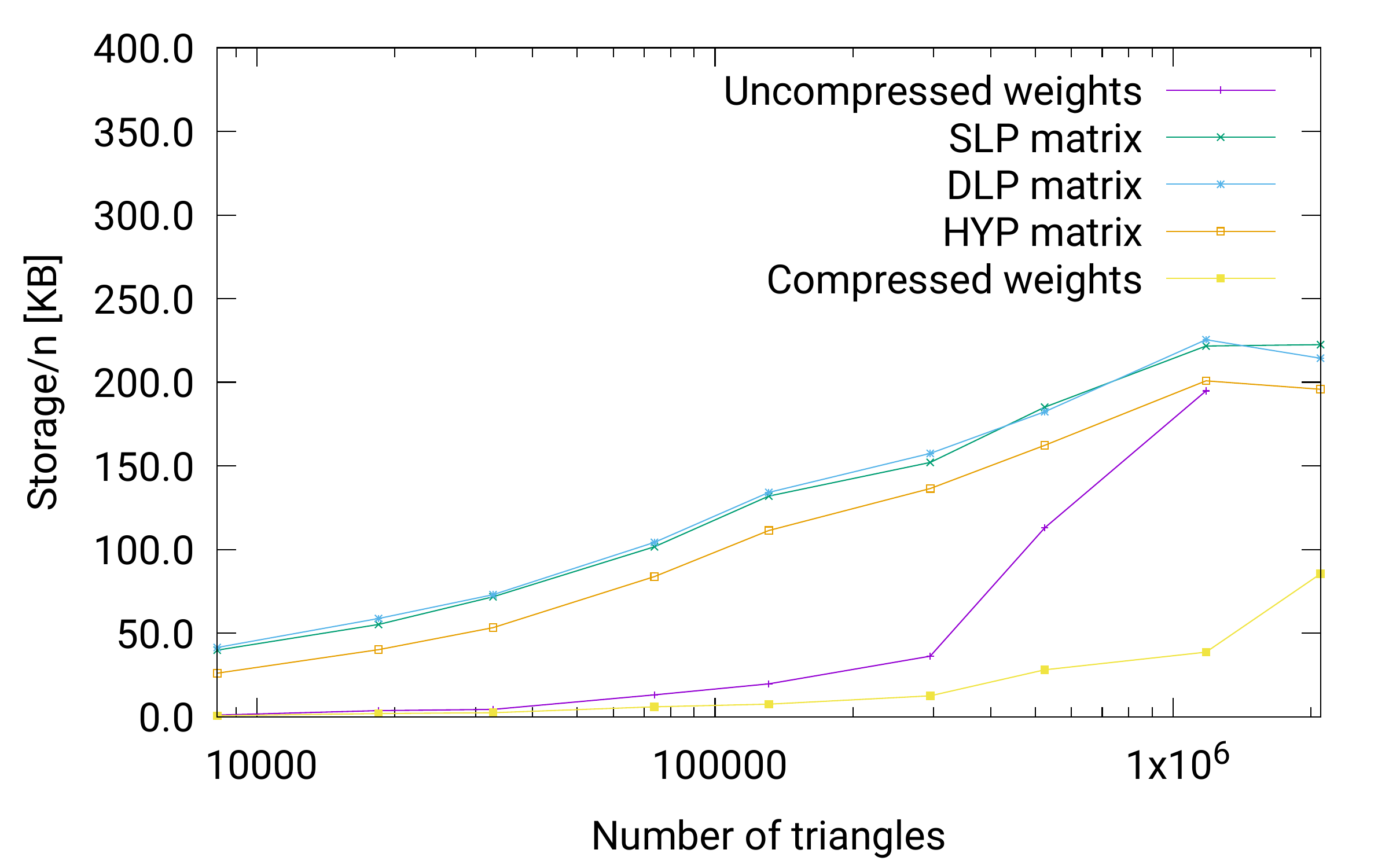}
  \end{center}
 
  \caption{Left: Runtime per degree of freedom.
     Right: Storage requirements per degree of freedom}
  \label{fi:hires}
\end{figure}

The results of our experiment show the expected increase both in computing
time and storage requirements: for $2\,097\,152$ triangles, the setup takes
far longer than in the low-frequency case, since more than thirteen times
as many blocks have to be considered.
In this case, a single coupling matrix requires $4$ MB of storage,
and storing multiple of these matrices during the setup of the matrix
$W_{sc}$ can be expected to exceed the capacity of the available cache
memory, thus considerably slowing down the computation.

Fortunately, the relatively long time required for setting up the
compressed weights is accompanied by a reduction in the time required
for setting up the $\mathcal{DH}^2$-matrices, since the smaller size
of the compressed weights compared to the original weights means that
the construction of the $\mathcal{DH}^2$-matrices has to work with
significantly smaller matrices and can therefore work faster.
This effect is not able to fully compensate the time spent compressing
the weight matrices, but it ensures that the new memory-efficient
algorithm is competitive with the original version.

Figure~\ref{fi:hires} illustrates the algorithms' practical performance:
the left figure shows the runtime per degree of freedom using a
logarithmic scale for the number of triangles.
We can see that the runtimes grow like $\mathcal{O}(n \log n)$, as predicted,
with the slope depending on the order of interpolation.

The right figure shows the storage, again per degree of freedom, and we
can see that the compressed $\mathcal{DH}^2$-matrices for the three operators
show the expected $\mathcal{O}(n \log n)$ behaviour.
The compressed weights grow slowly, while the uncompressed weights appear to
be set to surpass the storage requirements of the matrices they are used
to construct.

\bibliographystyle{plain}
\bibliography{hmatrix}

\end{document}